%% file: LAA-2020-9-Nov__Andr_s_Vargas_s_conflicted_copy_2020-11-11_.tex
\documentclass[letter,10pt,reqno]{amsart}
\usepackage[margin=4cm]{geometry}
\usepackage{amsmath,amsthm,amsfonts,amssymb,float}
\usepackage{yhmath,nicefrac,nccmath}
\usepackage{mathtools} 
\usepackage{graphicx,xcolor}
\usepackage{microtype}
\usepackage[noadjust]{cite}
\usepackage{booktabs}
\usepackage{dsfont} 
\usepackage[utf8]{inputenc} 
\usepackage[pdftex,bookmarksnumbered,colorlinks,plainpages]{hyperref}   
\hypersetup{citecolor=blue,linkcolor=purple}

\newtheorem{theorem}{Theorem}[section]
\newtheorem{corollary}[theorem]{Corollary}
\newtheorem{lemma}[theorem]{Lemma}
\newtheorem{proposition}[theorem]{Proposition}
\theoremstyle{definition}
\newtheorem{definition}[theorem]{Definition}
\theoremstyle{remark}
\newtheorem{remark}[theorem]{Remark}

\DeclareMathAlphabet\EuRoman{U}{eur}{m}{n}
\SetMathAlphabet\EuRoman{bold}{U}{eur}{b}{n}

\DeclareSymbolFont{extraitalic}      {U}{zavm}{m}{it}
\DeclareMathSymbol{\Qoppa}{\mathord}{extraitalic}{161}
\DeclareMathSymbol{\qoppa}{\mathord}{extraitalic}{162}
\DeclareMathSymbol{\Stigma}{\mathord}{extraitalic}{167}
\DeclareMathSymbol{\Sampi}{\mathord}{extraitalic}{165}
\DeclareMathSymbol{\sampi}{\mathord}{extraitalic}{166}
\DeclareMathSymbol{\stigma}{\mathord}{extraitalic}{168}

\DeclareRobustCommand{\anchor}{%
  \mathord{\text{\kern.1ex
    \ooalign{%
      \raisebox{0.1ex}{\scalebox{0.86}{$\dag$}}\cr
      \kern-.3ex\raisebox{-.15em}{\scalebox{0.8}{$\smallsmile$}}\cr%
  }}}%
}
      
\DeclareMathSymbol{\widehatsym}{\mathord}{largesymbols}{"62}

\makeatletter
\newcommand*\rel@kern[1]{\kern#1\dimexpr\macc@kerna}
\newcommand*\widebar[1]{%
  \begingroup
  \def\mathaccent##1##2{%
    \rel@kern{0.8}%
    \overline{\rel@kern{-0.8}\macc@nucleus\rel@kern{0.2}}%
    \rel@kern{-0.2}%
  }%
  \macc@depth\@ne
  \let\math@bgroup\@empty \let\math@egroup\macc@set@skewchar
  \mathsurround\z@ \frozen@everymath{\mathgroup\macc@group\relax}%
  \macc@set@skewchar\relax
  \let\mathaccentV\macc@nested@a
  \macc@nested@a\relax111{#1}%
  \endgroup
}
\makeatother

\newcommand{\comment}[1]{}

\renewcommand{\leq}{\leqslant}
\renewcommand{\geq}{\geqslant}

\newcommand{\R}{\mathds{R}}

\newcommand{\itm}[1]{\textup{(}\textit{#1}\textup{)}}

\newcommand{\I}{\mathrm{I}}
\newcommand{\II}{\mathrm{II}}
\newcommand{\III}{\mathrm{III}}
\newcommand{\Ia}{\mathrm{I}^{\mathrm{aff}}}

\newcommand{\IIIa}{\mathrm{III}^{\mathrm{aff}}}

\title{Blaschke's asymptotic lines of surfaces in $\mathbb{R}^3$}

\author{Martín Barajas-Sichacá} 
\author{Ronaldo Garcia}
\author{Andr\'es Vargas}

\date{}
 
\begin{document}

\begin{abstract}
In this paper we consider the Blaschke's asymptotic lines (also called affine asymptotic lines) of regular surfaces in 3-space. We study the binary differential equations defining Blaschke's asymptotic lines in the elliptic and hyperbolic regions of the surface near  affine cusp  points and to flat affine umbilic points. We also describe the affine asymptotic lines near  the   Euclidean parabolic set including the Euclidean flat umbilic points.
\\


\end{abstract}

\maketitle

\section{Introduction}

The geometry of surfaces is a classical subject in mathematics that has a long tradition of interesting results and whose range of applicability includes a wide spectrum of disciplines including theoretical and applied physics, material engineering and design, computer vision, and mathematical biology, among others. This broad range of applications stem from the simple fact that surfaces are ubiquitous in everyday life because we live immersed in a three-dimensional space.

Many techniques have been developed to study different aspects of the geometry of surfaces, including algebraic, analytic and topological tools. For example, according to Felix Klein's ``\textit{Erlangen Program}" (1872), the use of geometric transformation groups provide a fundamental method for the classification of different geometries. The study of properties of geometric objects that are invariant under a given transformation group $\mathcal{G}$ is called \textsl{the geometry subordinated to $\mathcal{G}$}. In particular, Euclidean geometry can be understood in this way when the group $\mathcal{G}$ is the group of so-called \emph{Euclidean motions}, i.e., translations and rotations in $\mathbb{R}^3$. Similarly, affine geometry is associated to the group of affine transformations, where in addition to translations and rotations, all invertible linear transformations are also admitted.

In this work we consider and study Blaschke's asymptotic lines defined by the third affine fundamental form in the context of affine differential geometry of surfaces in $\mathbb{R}^3$. The main goal will be to  obtain the generic singularities and to describe the qualitative behavior of these lines near them.

Recall that in the classical case of a surface $S\subset \mathbb{R}^3$, an asymptotic direction at $T_pS$ is a direction $v$ such that $v$ is tangent to a regular curve defined by the intersection of $S$ with $T_pS$.  In the elliptic part the asymptotic directions are imaginary, and in the hyperbolic part for each point we have two asymptotic directions. 
At parabolic points,  in the generic case (the parabolic set is formed of regular curves) and we have one  asymptotic direction which can be tangent or transversal to the parabolic curve.  For basic properties and historical aspects of asymptotic lines  see    \cite{Hilbert1952}  and \cite{struik}.
A generalization of the notion  of asymptotic directions and lines involves the concept of second fundamental form and conjugated directions.
A natural way is to consider   a congruence of lines in $\mathbb R^3$  defined by
\[ z(x,y)+t\phi(x,y)  \]
The second fundamental form of the congruence is defined by $\II=\langle dz, d\phi\rangle.$
The directions where $\II=0$ are called asymptotic directions of the congruence. See \cite[Appendix B]{sasaki}.

In the context of affine differential geometry the second fundamental is classicaly called the third affine fundamental form and it is defined as $\III=\langle D\nu, D\xi\rangle$, where $\nu$ is the conormal and $\xi$ is the affine normal of Blaschke. Further details appear in the preliminaries.

This paper is organized as follows. In Section~\ref{sec:preliminares} we provide a review of the concepts of fundamental forms in Euclidean and affine differential geometry of surfaces, including the differential equation of special curves, in particular, asymptotic lines. In Section~\ref{sec:conormal} the co-normal surface is introduced and the relation between affine asymptotic lines of the original surface and asymptotic lines of the co-normal surface is established. In Section~\ref{sec:section4} we consider the affine asymptotic lines near the singularities which are not parabolic Euclidean points. In Section~\ref{sec:section5} we analize the affine asymptotic lines near the parabolic Euclidean points. Finally, in Section~\ref{sec:section6} an example of the torus of revolution showing the global behavior of affine asymptotic lines is presented.

\subsection*{Acknowledgements}

The first author was supported by Pontificia Universidad Javeriana, research project ID-PROY 20165 during a post-doctoral fellowship at Department of Mathematics, Pontificia Universidad Javeriana, Bogot\'a, Colombia.
The third author also acknowledges partial support from the same project. 
The second author is fellow of CNPq and coordinator of the project PRONEX/FAPEG/CNPq.

%
%
\section{Preliminaries}\label{sec:preliminares} 
To provide a rigorous context to the statement of our results, we start with a brief  review of the required notions of classical differential geometry.
%
%
\subsection{Classic facts in Euclidean differential geometry}
Consider a smooth germ of an immersion $\alpha\colon S\to\mathbb{R}^3$ of a surface $S$ into Euclidean 3-space. In Euclidean differential geometry the \textit{fundamental forms} of $\alpha$ at point $p\in S$ are defined as the symmetric bilinear forms on the tangent space $T_pS$ given as follows (see~\cite{Carmo1976}):
\begin{itemize}
\item The \textit{first fundamental form}
\begin{equation}\label{EIff}
\I_\alpha(p;w_1,w_2)=\left\langle D\alpha(p;w_1),D\alpha(p;w_2)\right\rangle.
\end{equation}
\item The \textit{second fundamental form}
\begin{equation}\label{EIIff}
\II_\alpha(p;w_1,w_2)=-\left\langle DN_\alpha(p;w_1),D\alpha(p;w_2)\right\rangle.
\end{equation}
\end{itemize}
Here $\left\langle\cdot\, , \cdot \right\rangle$ is the Euclidean inner product on $\mathbb{R}^3$, the vectors $w_{1},w_{2}\in T_{p}S$,  and $N_\alpha$ is the unit normal associated to the immersion:
\begin{equation*}
N_\alpha=\frac{\alpha_u\wedge\alpha_v}{\left|\alpha_u\wedge\alpha_v\right|},
\end{equation*}
where $(u,v)\colon U\subset S\to\mathbb{R}^2$ denotes a chart on $S$, ``$\wedge$'' stands for the vector (or wedge) product in $\mathbb{R}^3$, and $\alpha_u\coloneqq\mfrac{\partial\alpha}{\partial u}$, $\alpha_v\coloneqq\mfrac{\partial\alpha}{\partial v}$. A vector $w\in T_pS$ for which the normal curvature
\begin{equation}\label{CurNor}
k_n(p;w)=\frac{\II_\alpha(p;w,w)}{\I_\alpha(p;w,w)}
\end{equation}
vanishes, is called an \textit{asymptotic direction of $\alpha$ at $p$}. A regular curve ${c\colon [a,b]\rightarrow S}$, whose tangent line is an asymptotic direction is called an \textsl{asymptotic line of $\alpha$}. Through every point $p$ of the hyperbolic region $\mathbb{H}_\alpha$ of the immersion $\alpha$; characterized by the condition that the Gaussian curvature $K_\alpha=\det DN_\alpha$ is negative, pass two transverse asymptotic lines of $\alpha$, tangent to the two asymptotic directions through $p$. When it is non-empty, the region $\mathbb{H}_\alpha$ is bounded by the set (generically a regular curve) $\mathbb{P}_\alpha$ of parabolic points of $\alpha$, on which $K_\alpha$ vanishes. On $\mathbb{P}_\alpha$, the pair of asymptotic directions degenerate into a single one. The parabolic points will be regarded here as the singularities of the asymptotic net. Along $\mathbb{P}_\alpha$, the asymptotic directions are transversal to the parabolic set except at isolated points called the \textit{cusp of Gauss points}. The asymptotic lines near to the  cusp of Gauss points have been well studied, see for example~\cite{Garcia1999,Izumiya2016} where it is shown that, generically, the asymptotic lines near to the cusp of Gauss points behave as in Fig.\,\ref{fig1}:
\begin{figure}[htb!]
	\centering
	\includegraphics[width=.9\textwidth,clip]{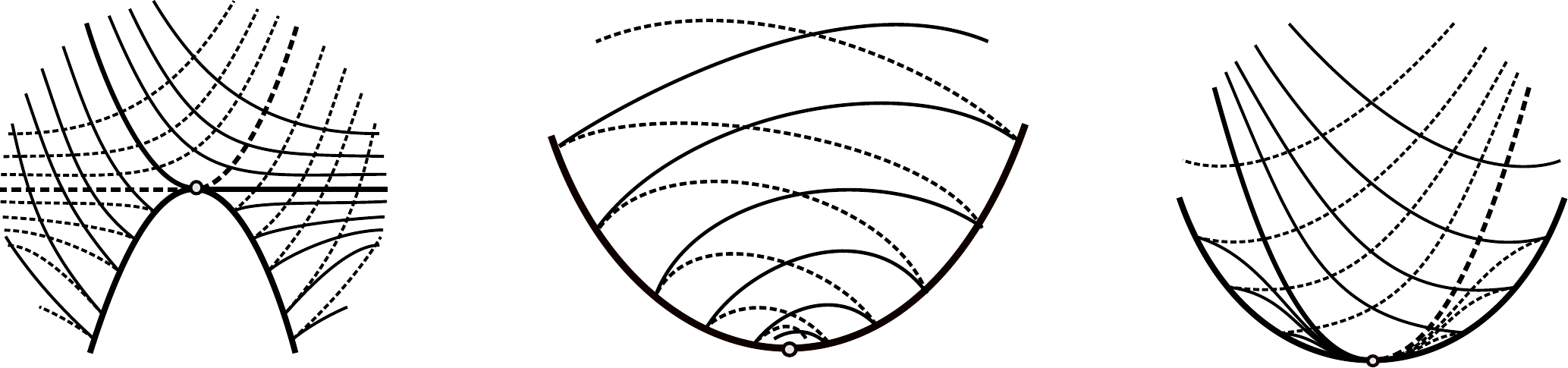}
	\caption{\small Asymptotic lines near Gauss cuspidal points: folded saddle (left), folded focus (center), and folded node (right).}
	\label{fig1}
\end{figure}

The closed  asymptotic lines and extended  closed   asymptotic lines were also studied in~\cite{Garcia1999}. Asymptotic lines, together with geodesics and principal curvature lines are studied in classical differential geometry by many authors. For principal curvature lines, which are smooth curves such that the tangent at each point is an eigenvector of $DN_\alpha$, there is a classification of  local topological models near the  umbilic points (see Fig.\,\ref{fig2}).
\begin{figure}[htb!]
	\centering
	\includegraphics[width=.9\textwidth,clip]{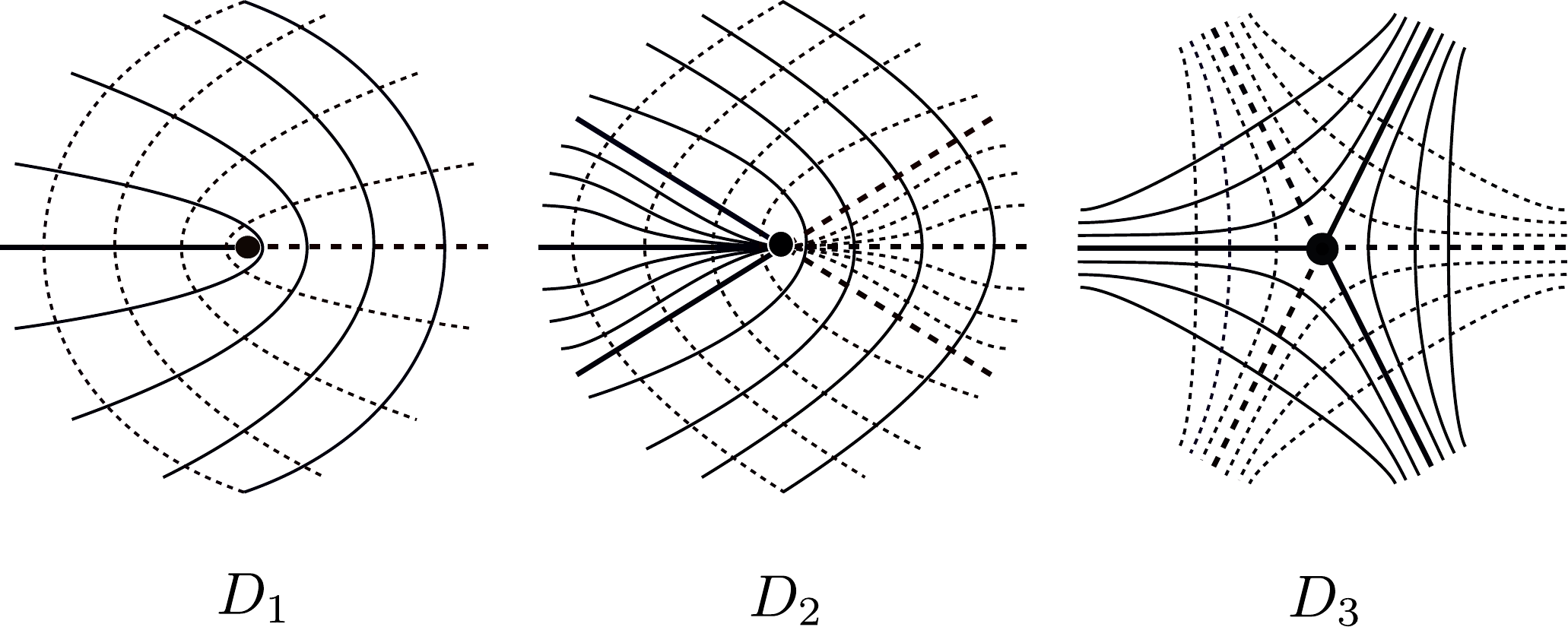}
	\caption{\small (Darbouxian umbilics). Generic and stable topological models of principal curvature lines near an umbilic point of index $1/2$ (left and center) and index $-1/2$ (right).}
	\label{fig2}
\end{figure}

Closed principal curvature lines were also studied, first by C.~Gutierrez and J.~Sotomayor in~\cite{GS-1982}. For more details on these results see~\cite{Garcia2009}.



\subsection{Affine differential geometry}
%
The affine differential geometry of surfaces is the study of properties of surfaces in three-dimensional space that are invariant under the group of unimodular affine transformations ASL($\mathbb{R}^3$). A survey about the origins of this subject of research can be fond in~\cite{Agnew2009}. Affine differential geometry has been studied by many authors and it is a subject of current research, see for example~\cite{BGC2020,Calabi1982,Davis2008,Davis2009,Simon,Nomizu1994,Buchin,Buchin1983}.

Classical invariants under rigid motions (Euclidean case) are heavily used in many applications of computer graphics and geometric modeling. The affine case, being much more general, allows to extend these tools to a broader set of situations and applications. For example, in~\cite{Andrade2011a,Andrade2012}  the authors have introduced affine differential geometry tools into the world of computer graphics.

Analogously to the Euclidean case, for a smooth germ of an immersion $\alpha\colon S\to\mathbb{R}^3$ of a surface $S$ in $3$-space, \textit{affine fundamental forms} of $\alpha$ at point $p\in S$ are also defined as symmetric bilinear forms on the tangent space $T_pS$ as we now explain (see~\cite{Buchin1983,Nomizu1994}).

The \textit{affine first fundamental form} or \textit{Berwald--Blaschke metric} is given by 
\begin{equation}\label{AIff}
\Ia_\alpha(p;w)=\left|K_\alpha\right|^{-\frac{1}{4}}\II_\alpha(p;w).
\end{equation}
If we write $w=a\alpha_u+b\alpha_v$ with $a,b\in\mathbb{R}$ then
\begin{equation*}
\Ia_\alpha(p;w) = a^2g_{11} +2ab\,g_{12} +b^2g_{22},
\end{equation*}
where
\begin{equation*}
g_{11}=\frac{L}{\left|LN-M^2\right|^{\frac{1}{4}}},\quad g_{12}=\frac{M}{\left|LN-M^2\right|^{\frac{1}{4}}},\quad g_{22}=\frac{N}{\left|LN-M^2\right|^{\frac{1}{4}}},
\end{equation*}
and
\begin{equation*}
L=\left|\alpha_u,\alpha_v,\alpha_{uu}\right|,\quad M=\left|\alpha_u,\alpha_v,\alpha_{uv}\right|,\quad N=\left|\alpha_u,\alpha_v,\alpha_{vv}\right|. 
\end{equation*}
Here $\left|w_1,w_2,w_3\right|$ denotes the determinant of the vectors $w_1,w_2,w_3$. Furthermore, the \textit{co-normal vector} $\nu$ to $S$ at $p$ is given by 
\begin{equation}\label{eq_conorm}
\nu(p)=\left|K_\alpha(p)\right|^{-\frac{1}{4}}N_\alpha(p)=\frac{1}{\left|LN-M^2\right|^{\frac{1}{4}}}\left(\alpha_u\wedge \alpha_v\right).
\end{equation}
There is a single transversal field $\xi$ defined on $S$, totally determined by the relations
\begin{equation}\label{eq_rel_noraff}
\left\langle \nu,\xi\right\rangle=1\quad\text{and}\quad\left\langle \xi,\nu_u\right\rangle=\left\langle \xi,\nu_v\right\rangle=0,
\end{equation}
 such that $\left\{\alpha_u,\alpha_v,\xi\right\}$ is a frame on the surface. The vector field $\xi$ is called \textit{affine normal field} and locally it is uniquely determined up to a direction sign. The affine normal vector satisfies $D\xi\subset TS$. Thus, for $p\in S$ and $w\in T_pS$, $D\xi(w)=B_{ij}w$, where the $B_{ij}$ is the \textsl{affine shape operator}. By the relations in~\eqref{eq_rel_noraff} and~\eqref{eq_conorm} we have
\begin{equation}\label{eq:na}
\xi=\frac{1}{\left|LN-M^2\right|^{\frac{1}{4}}}\left(\nu_u\wedge \nu_v\right).
\end{equation}
Denote by $(b_{ij})$ the coefficients of the matrix of affine shape operator in the basis $\left\{\alpha_u,\alpha_v\right\}$, thus
\begin{equation}\label{Der_norm_afin}
\left[
\begin{array}{c}
\xi_u \\
\xi_v 
\end{array}
\right]=
\left[
\begin{array}{cc}
b_{11} & b_{21} \\
b_{12} & b_{22}
\end{array}
\right]
\left[
\begin{array}{c}
X_u \\
X_v 
\end{array}
\right],
\end{equation}
where
\begin{align}\label{bij1}
b_{11} & = \left|LN-M^2\right|^{-\frac{1}{4}}\left|\xi_u,X_v,\xi\right|, \notag\\
b_{21} & = \left|LN-M^2\right|^{-\frac{1}{4}}\left|X_u,\xi_u,\xi\right|, \notag\\
b_{12} & = \left|LN-M^2\right|^{-\frac{1}{4}}\left|\xi_v,X_v,\xi\right|, \\
b_{22} & = \left|LN-M^2\right|^{-\frac{1}{4}}\left|X_u,\xi_v,\xi\right|.
\end{align} 
The affine third fundamental form is defined as 
\begin{equation}\label{AIIIff}
\IIIa_\alpha=\left\langle D\nu,D\xi\right\rangle,
\end{equation}
where $\langle\cdot\,,\cdot\rangle$ is the Euclidean inner product on $\mathbb{R}^3$. Again, if $w=a\alpha_u+b\alpha_v$ then
\begin{equation}\label{AIffcoord}
\IIIa_\alpha(w)=\left\langle D\nu(p;w),D\xi(p;w)\right\rangle = la^2+2mab+nb^2,
\end{equation}
where
\begin{equation}\label{AIffcoeff}
l=\left\langle \nu_u,\xi_u\right\rangle,\quad m=\left\langle \nu_u,\xi_v\right\rangle=\left\langle \nu_v,\xi_u\right\rangle,\quad n=\left\langle \nu_v,\xi_v\right\rangle.
\end{equation}
Using~\eqref{Der_norm_afin}, \eqref{bij1} and~\eqref{AIffcoeff} we obtain
\begin{align*}\label{eqbij1and3}
-l & = b_{11}g_{11}+b_{21}g_{12}, \\
-m & = b_{11}g_{12}+b_{21}g_{22}, \\
-m & = b_{12}g_{11}+b_{22}g_{22}, \\
-n & = b_{12}g_{12}+b_{22}g_{22}. 
\end{align*}
Thus are obtained the coefficients $b_{ij}$ in terms of the affine first and third fundamental forms as follows
\begin{equation}\label{eq_bijI_III}
\left[
\begin{array}{cc}
b_{11} & b_{21} \\
b_{12} & b_{22}
\end{array}
\right] = -\left|g_{11}g_{22}-g_{12}^2\right|^{-1}
\left[ 
\begin{array}{cc}
l & m \\
m & n
\end{array}
\right]
\left[
\begin{array}{cc}
g_{22}  & -g_{12} \\
-g_{12} & g_{11}
\end{array}
\right] 
\end{equation}
\begin{lemma}
Given $w\in T_pS$, 
\begin{equation*}
\IIIa_\alpha(w)=-\Ia_\alpha(w,D\xi(p;w)).
\end{equation*}
\end{lemma}
\begin{proof}
Let $\gamma\colon I\rightarrow S$ a smooth curve on such that $\gamma(0)=p$ and $\gamma'(0)=w$. With these conditions $D\xi(p,w)=\mfrac{d}{dt}(\xi\circ\gamma(t))|_{t=0}$ and using~\eqref{eq_bijI_III} we have
\begin{align*}
  \Ia_\alpha(w,D\xi(p,w))
  &= \Ia_\alpha\bigl(\alpha_uu'(0)+\alpha_vv'(0),\xi_uu'(0)+\xi_vv'(0)\bigr), \\
  &= \Ia_\alpha(\xi_u,\alpha_u)(u'(0))^2 
    + \Ia_\alpha(\xi_v,\alpha_v)(v'(0))^2  \\
  &\hspace{8em} + ( \Ia_\alpha(\xi_u,\alpha_v)+\Ia _\alpha(\xi_v,\alpha_u) )u'(0)v'(0) , \\
  &= (b_{11}g_{11}+b_{21}g_{12})(u'(0))^2 + (b_{12}g_{12}+b_{22}g_{12})(v'(0))^2 \\
  &\hspace{5em} + \bigl( b_{12}g_{11} +(b_{11}+b_{22})g_{12}+b_{21}g_{22} \bigr) u'(0)v'(0), \\
  &= -\left(l(u'(0))^2+2mu'(0)v'(0)+n(v'(0))^2\right), \\
  &= -\IIIa_\alpha(w),
\end{align*}
which completes the proof.
\end{proof}
 Affine curvature lines and affine asymptotic lines can be defined analogously to the Euclidean case. In~\cite{Buchin}, the author studied the conditions for which the affine curvature lines and Euclidean curvature lines are the same. In~\cite{BGC2020}, the authors studied affine curvature lines near to affine umbilic points and near  other singularities of the principal fields of affine principal directions. In particular, it is shown in~\cite{BGC2020} that near singularities, the principal affine field of affine directions has 17 generic different topological models, which is very different from the Euclidean case which only has three (Figure~\ref{fig2}) generic models. In~\cite{BG2018} the authors consider closed affine curvature lines.

The \emph{Gauss--Kronecker curvature} and the \emph{affine mean curvature} are defined, respectively, as 
\begin{equation*}
\mathcal{K}_{\alpha}^{\mathrm{aff}}=\det D\xi=b_{11}b_{22}-b_{12}b_{21},\quad\text{and}\quad\mathcal{H}_{\alpha}^{\mathrm{aff}}=-\frac{1}{2}\left(b_{11}+b_{22}\right).
\end{equation*}
By direct calculations using~\eqref{eq_bijI_III} we have
\begin{equation}\label{eqGKCandH}
  \mathcal{K}_{\alpha}^{\mathrm{aff}}=\frac{l\,n-m^2}{g_{11}g_{22}-g_{12}^2},\quad\text{and}\quad
  \mathcal{H}_{\alpha}^{\mathrm{aff}}=\frac{l\,g_{22}-2m\,g_{12}+n\,g_{11}}{g_{11}g_{22}-g_{12}^2}.
\end{equation}
\begin{definition}[\cite{Davis2008,Davis2009}]\label{def_elip_parab_hyp_aff}
A surface point is called an \emph{affine elliptic}, \emph{affine parabolic} or \emph{affine hyperbolic point} if the Gauss--Kronecker curvature is positive, equal to zero or negative at that point respectively.
\end{definition}
A definition of the affine normal curvature at a point $p\in S$ was provided in~\cite{Davis2008} (see also~\cite{Davis2009}) as follows
\begin{definition}\label{def_cur_nor_af}
The \emph{affine normal curvature} of $S$ at $p$ is given by
\begin{equation*}
k_n^{\mathrm{aff}}(p;w)=\frac{\IIIa_\alpha(p;w)}{\Ia_\alpha(p;w)},
\end{equation*}
where $w\in T_pS$ is not an asymptotic direction.
\end{definition}

Now we introduce two central notions for the study undertaken in this work.
\begin{definition}[\cite{Davis2008}]\label{def_dir_lin_asymo_af}
 Let $\alpha\colon S\to\mathbb{R}^3$ be a smooth immersion of a surface $S$ in $\R^3$.
\begin{enumerate}
\item A vector $w\in T_pS$ is an  \emph{affine asymptotic direction} (or a \emph{Blaschke asymptotic direction}) of the immersion $\alpha$ at $p\in S$ if $k_n^{\mathrm{aff}}(p;w)$ vanishes.
\item A regular curve ${c\colon I\rightarrow S}$ whose tangent line is an affine asymptotic direction is called an  \emph{affine asymptotic line} (or a \emph{Blaschke asymptotic line}) of $S$.
\end{enumerate}
 \end{definition}

\begin{proposition}[\cite{Davis2008,Davis2009}]\label{prop_dirasym}
We have $k_n^{\mathrm{aff}}(p;w)=0$ if, and only if $D^2\xi(p;w)\in T_pS$.
\end{proposition}
\begin{proof}
By~\eqref{eq_rel_noraff} we have $\left\langle D\xi(p;w),\nu\right\rangle=0$, thus 
\begin{equation}
 \left\langle D^2\xi(p;w),\nu\right\rangle = -\left\langle D\xi(p;w),D\nu(p;w)\right\rangle=-\IIIa_\alpha(p,w),
\end{equation}
and the result follows by Definition~\ref{def_dir_lin_asymo_af}~\itm{i}.
\end{proof}
\begin{corollary}\label{col_linasym}
  A regular curve ${c\colon I\rightarrow S}$ is an affine asymptotic line of $S$ if, and only if $\mfrac{d^2}{dt^2}\xi(t)\in T_{c(t)}S$, for all $t\in I$, where $\xi(t)=\xi|_{c(t)}$.
\end{corollary}
\begin{remark}\label{remedb}
By Definition~\ref{def_cur_nor_af}, $w=\alpha_u(p) u'(0)+\alpha_v(p) v'(0)\in T_pS$ is an affine asymptotic line of $S$ if, and only if $l(p)(u')^2+2m(p)u'v'+n(p)(v')^2=0$, which corresponds to a binary differential equation.
\end{remark}

We consider now the equation of the affine asymptotic lines.
\begin{lemma}\label{lm_elaa}
Let $\gamma(t)=\alpha(u(t),v(t))$ a smooth curve on $S$. Then, $\gamma$ is an affine asymptotic line of $S$ if, and only if, $\gamma$ is a solution of the binary differential equation
\begin{equation}\label{eqlaa}
l(\gamma(t))\left(u'(t)\right)^2+2m(\gamma(t))u'(t)v'(t)+n(\gamma(t))\left(v'(t)\right)^2=0.
\end{equation}
\end{lemma}
\begin{proof}
The result follows from Corollary~\ref{col_linasym} and Remark~\ref{remedb}.
\end{proof}
%
%
\subsection{Implicit differential equations}\label{sside}
In this section we consider the basic notions of the classical theory of differential equations not solved for derivatives. In this context, we will use tools of the theory of singularities and the geometry of space of jets.
\begin{definition}\label{defIDE}
An \textit{Implicit Differential Equation (IDE)} is an equation of the form
\begin{equation}\label{eqedi}
F(x,y,p)=0,\quad\text{where}\quad p=\frac{dy}{dx}.
\end{equation}
\end{definition}
The direction of the $p$-axis in the space of jets is called the \textit{vertical direction}. In the space of jets,  condition~\eqref{eqedi} defines a surface which we denote by $M$. The derivatives of the smooth function $F$ determine a multivalued direction field. Implicit differential equations have been studied by many authors, see for example~\cite{Arnold1983,Davydov}.
%
%
\subsubsection{Binary Differential Equations}\label{bde}
A special case of the family given in~\eqref{eqedi} are the so-called \textit{Binary Differential Equations (BDEs)}, which are equations of the form
\begin{equation}\label{eqedb}
F(x,y,p)=a(x,y)+2b(x,y)p+c(x,y)p^2=0,
\end{equation}
where $a,b,c$ are smooth functions. BDEs arise naturally in differential geometry, for example, lines of principal curvature, asymptotic lines, and characteristic lines are defined as the integral curves of their respective binary differential equations. BDEs have been well studied due to their interest for the qualitative theory of ordinary differential equations, and because of their role in singularity theory, among other applications. For more details see~\cite{Arnold1983,Arnold1985,BGC2020,BG2018,Bruce1989,Bruce1995,Dav1985,Davydov,Garcia1999,Garcia2009,GS-1982,Izumiya2016,Buchin}.\\

\begin{definition}[\cite{Arnold1983}]
  Let $M$ be the surface defined by~\eqref{eqedb} and consider the projection 
\begin{equation}
\pi\colon M\rightarrow\mathbb{R}^2,\qquad\pi(x,y,p)=(x,y).
\end{equation}
in the vertical direction. Associated to $\pi$ we have the following terminology:
\begin{enumerate}
\item A point of the surface $M$ is said to be \textit{regular} if it is not a critical point of the mapping $\pi$.
\item The set of critical values of the mapping $\pi$ is called the \textit{discriminant} curve.
\end{enumerate}
\end{definition}

Note that 
\begin{equation}
F(x,y,p)=0,\quad\text{implies}\quad DF=\frac{\partial F}{\partial x}+\frac{\partial F}{\partial y}\frac{dy}{dx}+\frac{\partial F}{\partial p}\frac{dp}{dx}=0.
\end{equation}
Thus the vector field given by
\begin{equation}\label{eqLieCartanField}
X(x,y,p)=\left(\dot{x},\dot{y},\dot{p}\right)=\left(\frac{\partial F}{\partial p},p\frac{\partial F}{\partial p},-\left(\frac{\partial F}{\partial x}+p\frac{\partial F}{\partial p}\right)\right)
\end{equation}
is tangent to $M$ at point $(x,y,p)$. It is well known that the projection onto the $xy$-plane turns the integral curves of~\eqref{eqedb} on $M$ restricted to a neighborhood of a regular point, into precisely the integral curves of the equation $\mfrac{dy}{dx}=v(x,y)$ restricted to a neighborhood of the projection of this point. 
\begin{definition}[\cite{Arnold1983}]\label{defsingset}
A point of the surface $M$ is said to be \textit{singular} for equation~\eqref{eqedb} if the projection $\pi$ of the surface to the plane is not a local diffeomorphism in the neighborhood of this point. The set of singular points of $F=0$ is called the \textit{criminant} of the equation.
\end{definition}
By Definition~\ref{defsingset} and the Implicit Function Theorem, the singular points on the surface $F=0$ are characterized by $\mfrac{\partial F}{\partial p}=0$.  
\begin{definition}[\cite{Arnold1983}]\label{defdiscrim}
The projection of the criminant onto the $xy$-plane parallel to the $p$-direction is called the \textit{discriminant curve}.
\end{definition}
Figure~\ref{fig3} shows the projection $\pi$ of $M$ onto the $xy$-plane, the criminant and discriminant sets, and the local behavior of the vector field $X$ in a neighborhood of a singular point.
\begin{figure}[htb!]
	\centering
	\includegraphics[width=.9\textwidth,clip]{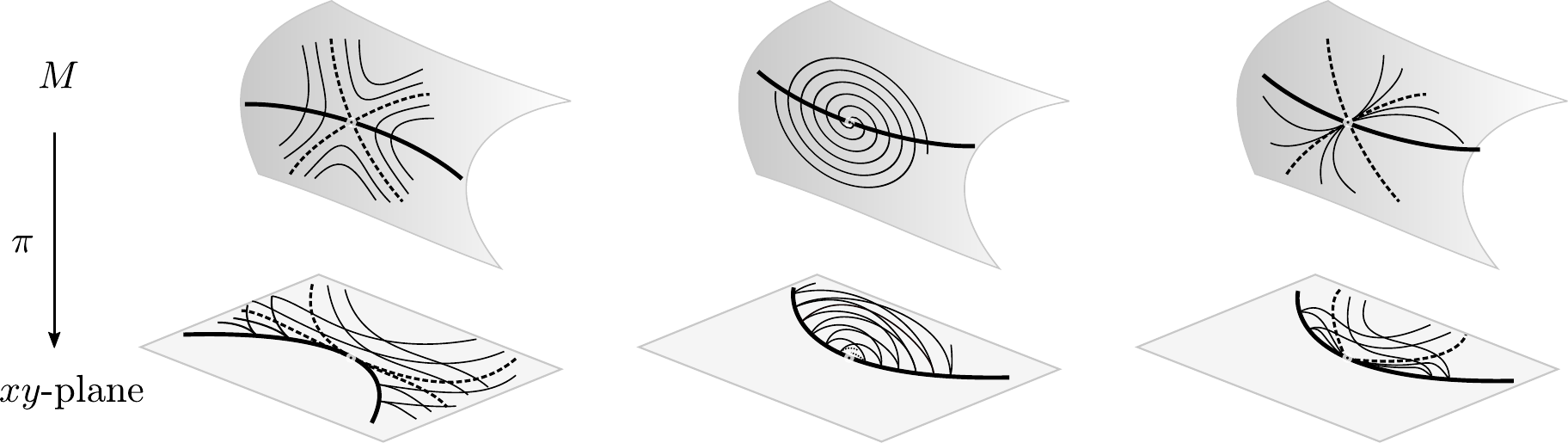}
	\caption{\small Integral curves of the vector field $X$ on $M$ with their respective projections to the $xy$-plane in a neighborhood of a singular point (see \cite{Arnold1983,Dav1985,Davydov}).}
	\label{fig3}
\end{figure}

%
%
\section{The co-normal surface}\label{sec:conormal}
In this section we consider the relation between the classical geometry of the co-normal surface, which we will define below and the affine differential geometry of the inital surface. 
\begin{definition}[{\cite[Chapter 2]{Nomizu1994}}]\label{defconsup}
Let $S$ a smooth surface and $\nu$ the co-normal vector on $S$ as defined in \eqref{eq_conorm}. We define the \textsl{Conormal Map} $\nu\colon S\rightarrow\mathbb{R}^3$ as the application defined by $p\mapsto\nu(p)$. We called the \textsl{Co-normal Surface} the set $\nu\left(S\,\backslash\,\mathbb{P}_{\alpha}\right)$ and denote it by $S^{\nu}$.
\end{definition}
\begin{proposition}[\cite{Nomizu1994}]\label{prop_con_sup}
The co-normal map $\nu$ is an immersion.
\end{proposition}
\begin{proof}
It is directly obtained by the relation in \eqref{eq_rel_noraff}.
\end{proof}
As a consequence of the Proposition~\ref{prop_con_sup}, the co-normal surface $S^{\nu}$ in locally a regular surface. The Theorems~\ref{th_rel_geo} and \ref{th_rel_geo_cupGauss} give a relationship between the $S^{\nu}$ differential geometry and the related affine differential geometry of $S$.
\begin{theorem}\label{th_rel_geo}
Let $\alpha\colon U\subset\mathbb{R}^2\rightarrow S$, $(u,v)\mapsto\alpha(u,v)$, a local parametrization of a smooth surface $S$ and  $S^{\nu}$ their corresponding co-normal surface. Consider the regular plane curve $\gamma\colon I\subset\rightarrow U$ in the domain of $\alpha$, $\gamma(t)=\left(u(t),v(t)\right)$. Then:
\begin{enumerate}
\item The curve $\beta=\alpha\circ\gamma$ is an affine asymptotic line of $S$ if and only if the curve $\eta=\nu\circ\gamma$ is an asymptotic line of $S^{\nu}$.
\item The point $p\in S$ is an affine parabolic point if and only if the point $\nu(p)\in S^{\nu}$ is a parabolic point.
\end{enumerate}
\end{theorem}
\begin{proof}
\begin{enumerate}
\item For this item we calculate the normal vector field and the second fundamental form of $S^{\nu}$ below. The normal field of $S^{\nu}$, denoted by $N_{\nu}$ is given by
\begin{equation*}
N_{\nu}=\frac{\nu_u\wedge\nu_v}{\left|\nu_u\wedge\nu_v\right|}=\lambda\,\xi,
\end{equation*}
where $\xi$ is the affine normal field of $S$ and $\lambda$ is a smooth function satisfying $\lambda\neq0$. The coefficients of the second fundamental form are given by
\begin{equation*}
e_{\nu}=\left\langle N_{\nu},\nu_{uu} \right\rangle=\lambda\left\langle \xi,\nu_{uu}\right\rangle=\lambda\, l,
\end{equation*}
\begin{equation*}
f_{\nu}=\left\langle N_{\nu},\nu_{uv} \right\rangle=\lambda\left\langle \xi,\nu_{uv}\right\rangle=\lambda\, m,
\end{equation*}
\begin{equation*}
g_{\nu}=\left\langle N_{\nu},\nu_{vv} \right\rangle=\lambda\left\langle \xi,\nu_{vv}\right\rangle=\lambda\, n.
\end{equation*}
It follows that along $\gamma$
\begin{align*}
  \II_{\nu}\left(\eta'\right) &= e_{\nu}\,du^2+2\,f_{\nu}\,du\,dv+f_{\nu}\,dv^2 \\
  &=\lambda\left(l\,du^2+2\,m\,du\,dv+n\,dv^2\right)=\lambda\,\IIIa_\alpha\left(\beta'\right),
\end{align*}
and as $\lambda\neq0$, the result follows.  
\item This item is a direct consequence of~\itm{i}.
\end{enumerate}
\end{proof}

\subsection{Height functions an affine height functions}

A geometrical characterization of an asymptotic line $\gamma$ on a smooth surface $S$, is that their osculating plane is tangent to $S$ along $\gamma$. It is well known that all the points of $\gamma$ have contact of order 2 with the tangent plane of $S$ at $\gamma(t)$. The order of contact of $S$ with its tangent plane at a cusp of Gauss point is equal to 3. The geometric information given via contact of a surface with planes and lines is called \textit{The flat geometry of $S$}. The contact of $S$ with planes is measured by the singularities of the height functions on $S$. The family of height functions $H\colon U\times\mathbb{S}^2\rightarrow\mathbb{R}$ on $S$ is given by
\begin{equation*}
H((u,v),\text{\bf{w}})=\langle \alpha(u,v),\text{\bf{w}}\rangle.
\end{equation*}
For $\bf{w}\in\mathbb{S}^2$ fixed, the height function $h_{\bf{w}}$ along the direction $\bf{w}$ is given by $h_{\bf{w}}(u,v)=H((u,v),\bf{w})$. We say that a map germ $\Gamma\colon \mathbb{R}^2\rightarrow\mathbb{R}$ has an $A^{\pm}_{k}$ singularity at $(0,0)$ if it is possible, via a change of coordinates and multiplying by constant terms to put $\Gamma$ as $x^2\pm y^{k+1}$. Table~\ref{table1} shows the geometrical characterization of the $A^{\pm}_{k}$ type singularities of $h_{\bf{w}}$ at $p=\alpha(0,0)$.
\begin{table}[ht!!!]
\caption{Geometric characterization of the local singularities of $h_{\bf{w}}$ (see~\cite{Izumiya2016}).}
\begin{center}
\begin{tabular}{@{}ccl@{}}
\toprule
Type & Normal form     &   Geometric characterization\\
\midrule
$A^{\pm}_{1}$ & $x^2\pm y^2$ & $\bf{w}$ is the normal direction to $S$ and $p$ is non a parabolic point. \\
$A_{2}$       & $x^2+y^3$    & $p$ is a parabolic point and their asymptotic direction is \\
              &              & transversal to the parabolic set.\\
$A^{\pm}_{3}$ & $x^2\pm y^4$ & $p$ is a parabolic point and their asymptotic direction is \\
              &              & tangent to the parabolic set.\\
\bottomrule
\end{tabular}
\end{center}
\label{table1}
\end{table}

In the affine differential geometry context, there is also the \textsl{affine height functions}, see for example \cite{Blaschke1923,Davis2008,Nomizu1994,Buchin1983}. The family of affine height functions $H^a\colon U\times\mathbb{S}^2\rightarrow\mathbb{R}$ on $S$ is given by
\begin{equation}\label{eq_aff_supp}
H^a((u,v),\text{\bf{w}})=\langle \nu(u,v),\text{\bf{w}}\rangle,\quad\text{where $\nu$ is the co-normal vector of $S$}.
\end{equation}
In \cite{Davis2008} the author studied the singularities of the members $h^a_{\bf{w}}$ of the family of affine height functions, (defined analogously as the euclidean case for a fixed $\bf{w}\in\mathbb{S}^2$) and characterized them geometrically (see Table~\ref{table2}).
\begin{table}[ht!!!]
\caption{Geometric characterization of the local singularities of $h^a_{\bf{w}}$ (see~\cite[Chapter 7]{Davis2008}).}
\begin{center}
\begin{tabular}{@{}ccl@{}}
\toprule
Type & Normal form  & Geometric characterization\\
\midrule
$A^{\pm}_{1}$ & $x^2\pm y^2$ & $\bf{w}$ is in the direction of the affine normal direction to $S$ \\
						&              & and $p$ is non an affine parabolic point. \\
$A_{2}$       & $x^2+y^3$    & $p$ is an affine parabolic point and their affine asymptotic  \\
              &              & direction is transversal to the affine parabolic set.\\
$A^{\pm}_{3}$ & $x^2\pm y^4$ & $p$ is an affine  parabolic point and their affine asymptotic  \\
              &              & direction is tangent to the affine parabolic set.\\
\bottomrule
\end{tabular}
\end{center}
\label{table2}
\end{table}

A special fact here is that the members of the family of affine height functions capture the flat geometry of $S^{\nu}$. Thus using the geometric information in the Tables~\ref{table1} and \ref{table2} we can conclude that:
\begin{theorem}\label{th_rel_geo_cupGauss}
In the conditions of Theorem \ref{th_rel_geo}, a point $p\in S$ is an affine cusp of Gauss point if and only if the point $\nu(p)\in S^{\nu}$ is a cusp of Gauss point.
\end{theorem}
For more details on geometric contact from a singularity theory viewpoint see~\cite{Arnold1985,Izumiya2016}.

%
%
\section{Blaschke's asymptotic lines away of Euclidean parabolic points}
\label{sec:section4}

It is well known that the Euclidean Gaussian curvature, denoted here by $K^{\mathrm{e}}$ for emphasis, induces a natural stratification on a regular surface $S$. Let $p\in S$, then it satisfies one of the conditions listed below:
\begin{itemize}
	\item if $K^{\mathrm{e}}(p)<0$, then $p$ is a \emph{hyperbolic point},
	\item if $K^{\mathrm{e}}(p)=0$, then $p$ is a \emph{parabolic point},
	\item if $K^{\mathrm{e}}(p)>0$, then $p$ is an \emph{elliptic point}.
\end{itemize}
In this paper we concentrate on two cases: the first is when $p$ is not a parabolic point and the other is when $p$ is exactly a parabolic point. In this section we consider the first case, i.e., $p$ is a point on $S$ away of the Euclidean parabolic set. With these conditions we have suitable charts for our study.
\begin{proposition}[Pick normal forms,~\cite{Davis2008,Buchin1983}]\label{prop_pick}
Let $S$ be a smooth surface locally parametrized by $\alpha\left(u,v\right)=\left(u,v,h\left(u,v\right)\right)$ and let $p=X(0,0)$. Assume that $p$ is not a parabolic point, then:
\begin{enumerate}
\item If $p$ is an elliptic point, $h\left(u,v\right)$ can be written as
\begin{multline*}
\quad h(u,v) = \frac{1}{2}(u^2+v^2)+\frac{1}{6}\sigma (u^3-3uv^2) + \frac{1}{24}(q_{40}u^4+4q_{31}u^3v \\
 + 6q_{22}u^2v^2+4q_{13}uv^3+q_{04}v^4)+\frac{1}{120}(q_{50}u^5+5q_{41}u^4v \\
 + 10q_{32}u^3v^2+10q_{23}u^2v^3+5q_{14}uv^4+q_{05}v^5) + O(6).
\end{multline*}
\item If $p$ is a hyperbolic point, $h\left(u,v\right)$ can be written as
\begin{multline*}
\quad h(u,v)  = \frac{1}{2}(u^2-v^2) + \frac{1}{6}\sigma (u^3+3uv^2) + \frac{1}{24}(q_{40}u^4+4q_{31}u^3v \\
+ 6q_{22}u^2v^2 + 4q_{13}uv^3 + q_{04}v^4) + \frac{1}{120}(q_{50}u^5 + 5q_{41}u^4v \\
+ 10q_{32}u^3v^2 + 10q_{23}u^2v^3 + 5q_{14}uv^4+q_{05}v^5)+O(6).
\end{multline*}
\end{enumerate}
\end{proposition}
\begin{remark}\label{rem3.2}
The Pick normal forms are obtained using a change of coordinates in the source and affine transformations in the target. Since an affine transformation is not necessarily an Euclidean isometry, the Pick normal forms are not Euclidean normal forms. More precisely, the expression appearing in Proposition~\ref{prop_pick}\,\itm{i} is not a local parametrization of an Euclidean umbilic point ($p=X(0,0)$).
\end{remark}
We will adopt the normal form in~\eqref{n_f_lip-hyp} below for the analysis of the non-parabolic case, 
\begin{multline}\label{n_f_lip-hyp}
 h(u,v) = \frac{1}{2}(u^2+\varepsilon v^2) + \frac{1}{6}\sigma(u^3-3\varepsilon uv^2 ) +  \frac{1}{24}(q_{40}u^4 + 4q_{31}u^3v \\
 + 6q_{22}u^2v^2 + 4q_{13}uv^3 + q_{04}v^4) + \frac{1}{120}(q_{50}u^5 + 5q_{41}u^4v \\
 + 10q_{32}u^3v^2 + 10q_{23}u^2v^3 + 5q_{14}uv^4 + q_{05}v^5) + \frac{1}{720}(q_{60}u^6 \\
 + 6q_{51}u^5v + 15q_{42}u^4v^2 + 20q_{33}u^3v^3 + 15q_{24}u^2v^4 + 6q_{15}uv^5 \\
 + q_{06}v^6) + \frac{1}{5040}(q_{70}u^7 + 7q_{61}u^6v + 21q_{52}u^5v^2 + 35q_{43}u^4v^3 \\
 + 35q_{34}u^3v^4 + 21q_{25}u^2v^5 + 7q_{16}uv^6 + q_{07}v^7) + O(8), \ 
\end{multline}
where, for the elliptic case $\varepsilon=1$, and for the hyperbolic case $\varepsilon=-1$, as in Proposition~\ref{prop_pick}.

Let $\alpha$ be a local parametrization as in Proposition~\ref{prop_pick}, then the coefficients of Blaschke's metric (affine first fundamental form) are given by

\begin{align*}
L &= 1 + \sigma u + \frac{1}{2}q_{40}u^2 + q_{31}uv + \frac{1}{2}q_{22}v^2 + \frac{1}{6}q_{50}u^3 \\
&\qquad + \frac{1}{2}q_{41}u^2v + \frac{1}{2}q_{32}uv^2 + \frac{1}{6}q_{23}v^3 + O(4), \\
M &= -\sigma\varepsilon v + \frac{1}{2}q_{31}u^2+q_{22}uv + \frac{1}{2}q_{13}v^2 + \frac{1}{6}q_{41}u^3 + \frac{1}{2}q_{32}u^2v \\
&\qquad  + \frac{1}{2}q_{23}uv^2 + \frac{1}{6}q_{14}v^3+O(4), \\
N &= \varepsilon-\sigma\varepsilon u+\frac{1}{2}q_{22}u^2+q_{13}uv+\frac{1}{2}q_{04}v^2+\frac{1}{6}q_{32}u^3+\frac{1}{2}q_{23}u^2v \\
&\qquad + \frac{1}{2}q_{14}uv^2+\frac{1}{6}q_{05}v^3+O(4),
\end{align*}
Now, in this chart the coefficients of the affine third fundamental form are given by
\begin{align}\label{eql1}
 l &= -\frac{1}{2}\sigma^2 + \frac{1}{4}q_{40} + \frac{1}{4}\varepsilon q_{22} + \biggl(\frac{1}{2}\sigma^3  - \sigma q_{40} + \frac{1}{4}q_{50} + \frac{1}{2}\sigma\varepsilon q_{22} \notag \\
 &\qquad + \frac{1}{4}\varepsilon q_{32}\biggr)u + \biggl(-\frac{1}{4}q_{31}\sigma + \frac{1}{4}q_{41} \frac{1}{4}\sigma\varepsilon q_{13} + \frac{1}{4}q_{23}\varepsilon\biggr) v + l_{20}u^2 \qquad \\
 &\qquad + l_{11}uv + l_{02}v^2+O(3), \notag
\end{align}
\begin{align}\label{eqm1}
 m &= \frac{1}{4}q_{31} + \frac{1}{4}\varepsilon q_{13} + \biggl(\frac{1}{4}q_{31}\sigma  + \frac{1}{4}q_{41} + \frac{3}{4}\sigma\varepsilon q_{13} + \frac{1}{4}q_{23}\varepsilon\biggr)u \notag \\
&\qquad + \left(-\frac{1}{2}\varepsilon\sigma^3+\sigma q_{22}+\frac{1}{4}q_{32}+\frac{1}{2}\varepsilon q_{04}\sigma+\frac{1}{4}\varepsilon q_{14}\right)v+m_{20}u^2+ r \\
&\qquad + m_{11}uv+m_{02}v^2+O(3), \notag
\end{align}
\begin{align}\label{eqn1}
 n &= -\frac{1}{2}\varepsilon\sigma^2 + \frac{1}{4}q_{22} + \frac{1}{4}\varepsilon q_{04} + \biggl(-\frac{1}{2}\varepsilon\sigma^3 + \frac{1}{4}\varepsilon q_{40}\sigma + \sigma q_{22} + \frac{1}{4}q_{32} \notag \\
&\qquad + \frac{1}{4}\varepsilon q_{04}\sigma + \frac{1}{4}\varepsilon q_{14}\biggr)u + \biggl(\frac{1}{4}\varepsilon q_{31}\sigma + \frac{7}{4}\sigma q_{13} + \frac{1}{4}q_{23} + \frac{1}{4}\varepsilon q_{05}\biggr)v  \\
&\qquad + n_{20}u^2+n_{11}uv + n_{02}v^2 + O(3), \notag 
\end{align}
where $l_{20},l_{11},l_{02},m_{20},m_{11},m_{02},n_{20},n_{11},n_{02}$ depend only on the coefficients of terms up to order five in $h(u,v)$.
\begin{lemma}\label{lem_discrim_parab}
The critical set of the BDE~\eqref{eqlaa} coincides with the affine parabolic set.
\end{lemma}
\begin{proof}
In Subsection~\ref{sside} we saw that the singular set of the BDE~\eqref{eqlaa} is characterized by 
\begin{equation*}
F(x,y,p)=F_p(x,y,p)=0,
\end{equation*}
where $F_p=\mfrac{\partial F}{\partial p}$. It follows that
\begin{equation*}
F(x,y,p)=l+2mp+np^2=0\quad\text{and}\quad F_p(x,y,p)=2\left(m+np\right)=0.
\end{equation*}
If $n\neq0$, we take $p=-\mfrac{m}{n}$ and replace into $F$ to obtain the result. When $n=0$, we necessarily have $m=l=0$, and in particular, $l\,n-m^2=0$.
\end{proof}
\begin{remark}\label{rem_regions}
If we consider the BDE~\eqref{eqlaa} as a quadratic equation in $p$ with coefficients $l$, $m$, and $n$ we have
\begin{equation*}
p=\frac{dv}{du}=\frac{-m\pm\sqrt{m^2-ln}}{n}.
\end{equation*}
Thus, for $q\in S$ satisfying $(m^2-ln)(q)>0$ there are two transversal affine asymptotic directions at $q$. If $q$ is a parabolic point there is an unique affine asymptotic direction with multiplicity two. If $(m^2-ln)(q)<0$ do not exist affine asymptotic directions. 
\end{remark}
We denote the affine elliptic and hyperbolic regions respectively by $\mathbb{E}^{a}_{\alpha}$ and $\mathbb{H}^{a}_{\alpha}$. The affine parabolic set will be denoted by $\mathbb{P}^{a}_{\alpha}$.
\begin{proposition}\label{prop_regions}
Let $\gamma\colon I\rightarrow S$ an affine asymptotic line of $S$. Then
\begin{enumerate}
\item If $\gamma$ is totally contained in the elliptic region, then $\gamma$ is totally contained in $\mathbb{H}^{a}_{\alpha}$.
\item If $\gamma$ is totally contained in $\mathbb{H}_{\alpha}$, then $\gamma$ is totally contained in $\mathbb{E}^{a}_{\alpha}$.
\end{enumerate}
\end{proposition}
\begin{proof}
It is sufficient to see that if $0\in I$ and $p=\gamma(0)$ is an elliptic point, $\Ia_{\alpha}(p)$ is positive definite, thus $\left(g_{11}g_{22}-g_{12}^2\right)(p)>0$. With these conditions $\mathcal{K}_{\alpha}^{\mathrm{aff}}(p)<0$ and it follows that $p\in\mathbb{H}^{a}_{\alpha}$. The second case is analogous.
\end{proof}
\begin{remark}\label{rem_symp_elliptic}
The Proposition~\ref{prop_regions} establish that the asymptotic net lives in the intersection of 
$$\mathbb{E}_{\alpha}\cap\left(\mathbb{H}^{a}_{\alpha}\cup\mathbb{P}^{a}_{\alpha}\right), \quad\text{or}\quad\mathbb{H}_{\alpha}\cap\left(\mathbb{E}^{a}_{\alpha}\cup\mathbb{P}^{a}_{\alpha}\right)$$
for the elliptic or hyperbolic cases, respectively, see Figure~\ref{fig4}. This behavior is very different to the Euclidean case for which the asymptotic net is defined in the topological closure of the hyperbolic region.
\begin{figure}[htb!]
	\centering
	\includegraphics[width=.6\textwidth,clip]{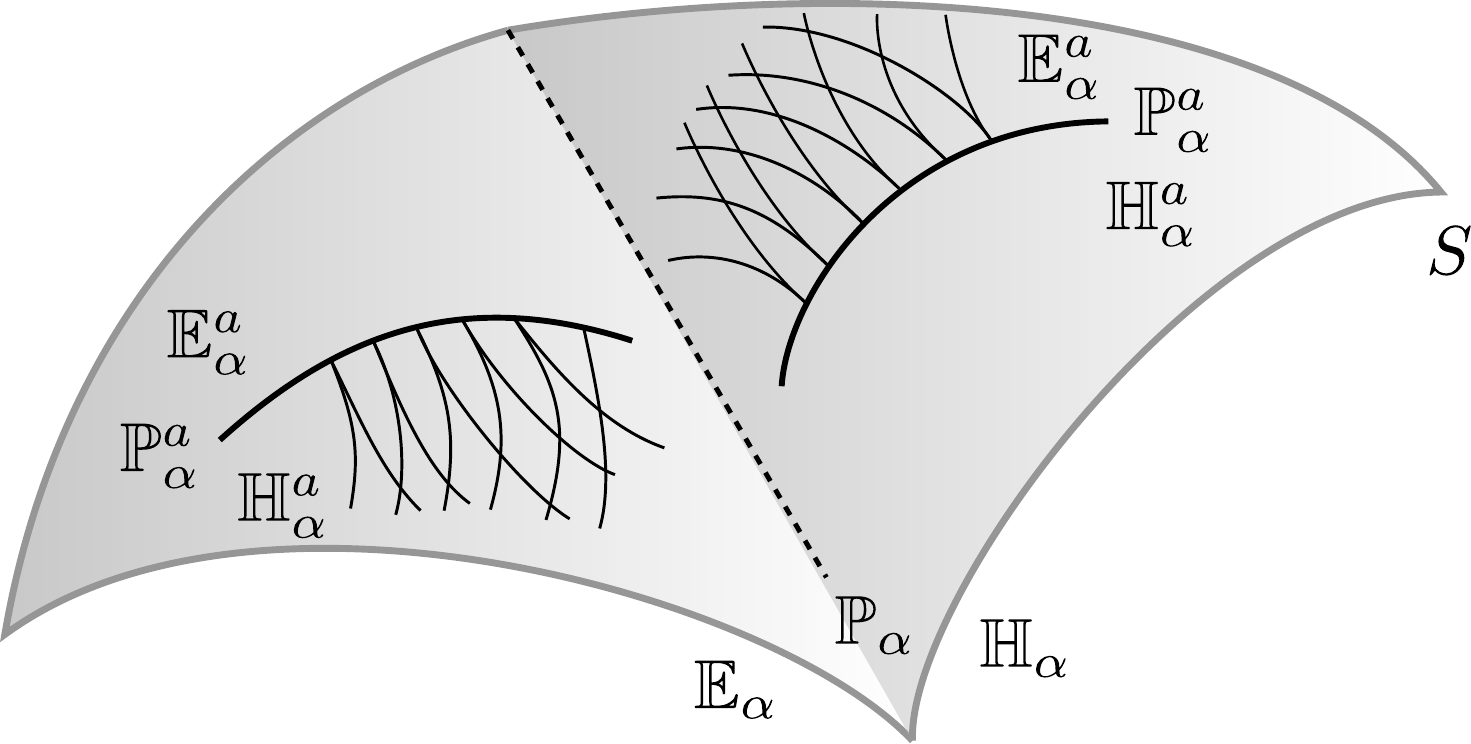}
	\caption{\small Natural stratification of $S$ in terms of the Gaussian curvature. In each one of the regions ($\mathbb{E}_{\alpha}$ and $\mathbb{H}_{\alpha}$) there exist other stratification induced by the Gauss-Kronecker curvature, there are $\mathbb{E}^{a}_{\alpha}$ and $\mathbb{H}^{a}_{\alpha}$ separated by the affine parabolic set $\mathbb{P}^{a}_{\alpha}$.
	}
	\label{fig4}
\end{figure}
\end{remark}
Analogous to Euclidean case, the affine parabolic set is generically a smooth curve. On the parabolic set there exist isolated special points which we will characterize below.
\begin{lemma}\label{lem_transv_aad}
In an ordinary affine parabolic point, the unique affine asymptotic direction is transversal to the affine parabolic set.
\end{lemma}
\begin{proof}
We consider the case when the origin is a regular affine parabolic point. Note that
\begin{align*}
 l(0,0) &= \frac{1}{4}\varepsilon(-2\varepsilon\sigma^2+\varepsilon q_{40}+q_{22}), \\
 m(0,0) &= \frac{1}{4}\varepsilon(\varepsilon q_{31}+q_{13}), \\
 n(0,0) &= -\frac{1}{4}\varepsilon(2\varepsilon^2\sigma^2-\varepsilon q_{22}-q_{04}).
\end{align*}
The point $p=\alpha(0,0)$ is an affine parabolic point if $(ln-m^2)(0,0)=0$. We can to choose
\begin{equation*}
q_{22}=-\varepsilon(-2\sigma^2+q_{40})\quad\text{and}\quad q_{31}=-\varepsilon q_{13},
\end{equation*}
thus $l(0,0)=m(0,0)=0$. In consequence,  
\begin{equation*}
l(0,0)(u'(0))^2+2m(0,0)u'(0)v'(0)+n(0,0)(v'(0))^2=n(0,0)(v'(0))^2=0,
\end{equation*}
then if $n(0,0)\neq0$ the affine asymptotic direction is determined by $v'(0)=0$. The case $n(0,0)=0$ will be considered in the Subsection~\ref{subaacfumb}. The tangent of affine parabolic set at $p$ is given by 
\begin{equation*}
\frac{1}{16}(q_{40}-q_{04})\bigl(\varepsilon(6\sigma^3+\varepsilon q_{32}-6q_{40}\sigma+q_{50})u'(0)+(\varepsilon q_{41}+2q_{13}\sigma+q_{23})v'(0)\bigr).
\end{equation*}
The case $q_{40}-q_{04}$ will be considered in Subsection~\ref{subaacfumb}. Generically, it holds that  $\varepsilon q_{41}+2q_{13}\sigma+q_{23}\neq0$ and the result follows.
\end{proof}
\begin{definition}\label{def_acg}
A point $p$ in $\mathbb{P}^{a}_{\alpha}$ is called an \textit{affine cusp of Gauss} if the affine asymptotic direction at this point is tangent to the affine parabolic set at $p$.
\end{definition}
\begin{remark}\label{rem_cusp}
It is proved in~\cite{Bruce1984} that the integral curves of BDEs along the discriminant set are diffeomorphic to ordinary cusps except at special points where the integral curve is tangent to the singular set. In our context, at ordinary affine parabolic points the affine asymptotic lines are diffeomorphic to ordinary cusps except at an affine cusp of Gauss, see Figure~\ref{fig5}.
\begin{figure}[htb!]
	\centering
	\includegraphics[width=.3\textwidth,clip]{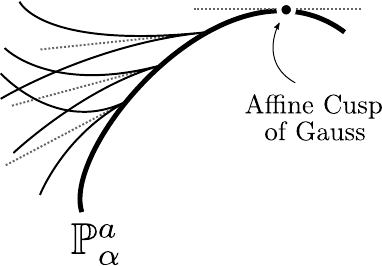}
	\caption{\small Affine asymptotic lines along the affine parabolic set are diffeomorphic to ordinary cusps except at affine cusp of Gauss points.}
	\label{fig5}
\end{figure}
\end{remark}
By Remark~\ref{rem_cusp} we will focus our attention on the dynamical behavior of affine asymptotic lines near to affine cusp of Gauss points.
\begin{proposition}\label{eqaaleh}
Consider the binary differential equation given by
\begin{equation}\label{eqedb1}
l\,du^2+2m\,dudv+n\,dv^2=0,
\end{equation}
where $l$, $m$, $n$ are as in~\eqref{eql1},~\eqref{eqm1},~\eqref{eqn1}. Suppose that $p=\alpha(0,0)$ is an affine cusp of Gauss. Then, by changes of coordinates~\eqref{eqedb1} can be written as
\begin{equation}\label{eq_edb_dav}
\left(-v+\lambda u^2+O(3)\right)du^2+\left(1+\Lambda u^2+O(3)\right)dv^2=0,
\end{equation}
where
\begin{multline*}
\lambda = -\frac{1}{\sigma(\varepsilon q_{41}+2q_{13}\sigma+q_{23})^2}\bigl((6\sigma^3+20q_{13}q_{41}\sigma^2 + (10q_{23}q_{41}-6q_{04})\sigma+q_{50})\varepsilon  \\
 + 20q_{13}^2\sigma^3+20q_{13}q_{23}\sigma^2+5(q_{23}^2+q_{41}^2)\sigma+q_{32}\bigr),
\end{multline*}
and $\Lambda$ has a big expression in terms of the coefficients of the parametrization $\alpha$.
\end{proposition}
\begin{proof} 
In local coordinates and using Lemma~\ref{lem_transv_aad} we have that if $q_{22}=-\varepsilon(-2\sigma^2+q_{40})$ and $q_{31}=-\varepsilon q_{13}$, then $p=\alpha(0,0)$ is an affine parabolic point. The condition for $p$ to be an affine cusp of Gauss is 
\begin{equation*}
q_{40}=\frac{1}{6}\frac{6\sigma^3+q_{50}+\varepsilon q_{32}}{\sigma}.
\end{equation*}
Under the stated hypotheses, the coefficients of~\eqref{eqedb1} are given by
\begin{align}\label{eqlmnfndav}
l(u,v) &= a_{01}v+a_{20}u^2+a_{11}uv+a_{02}v^2+O(3), \\
m(u,v) &= b_{10}u+b_{01}v+b_{20}u^2+b_{11}uv+b_{02}v^2+O(3), \\
n(u,v) &= c_{00}+c_{10}u+c_{01}v+c_{20}u^2+c_{11}uv+c_{02}v^2+O(3),
\end{align}
where $a_{ij},b_{kl},c_{st}$ depend on the coefficients of the parametrization. In particular, 
\begin{equation}\label{eqa00}
c_{00}=-\frac{1}{24\sigma}\varepsilon(6\sigma^3-6q_{04}\sigma+q_{50}+\varepsilon q_{32}).
\end{equation}
Generically $c_{00}\neq0$, then without loss of generality we can divide~\eqref{eqedb1} by the expression in~\eqref{eqa00} to obtain a new equation where the constant term is equal to one. We will continue with the same notation for the coefficients given in~\eqref{eqlmnfndav} but now $c_{00}=1$. The 2-jet of the discriminant curve is given by $v=\mfrac{b_{10}^2a_{20}}{a_{01}}u^2$, thus at $(0,0)$ the vector $(1,0)$ is tangent to the affine parabolic set. Now consider the application $\psi\colon (x,y)\mapsto(u(x,y),v(x,y))$ where
\begin{align*}
u(x,y) &=s_{10}x+s_{01}y+s_{20}x^2+s_{11}xy+s_{02}y^2+s_{30}x^3+s_{21}x^2y+s_{12}xy^2+s_{03}y^3+\ldots , \\
v(x,y) &=t_{10}x+t_{01}y+t_{20}x^2+t_{11}xy+t_{02}y^2+t_{30}x^3+t_{21}x^2y+t_{12}xy^2+t_{03}y^3+\ldots,
\end{align*}
with $s_{10}t_{01}-s_{01}t_{10}\neq 0$. $\psi$ is a germ of a diffeomorphism at $(0,0)$ and we used it for a change of coordinates in~\eqref{eqedb1}. It follows that for a suitable choice of the coefficients of $\psi$ we obtain that the BDE takes the desired form as in~\eqref{eq_edb_dav}, where
\begin{equation*}
\lambda=-\frac{1}{2}\frac{a_{01}b_{10}+2b_{10}^2-2a_{20}}{a_{01}^2}
\end{equation*}
and
\begin{multline*}
\Lambda = \frac{1}{32a_{01}^3}\bigl((-4b_{01}^2-4b_{01}c_{10}-16b_{10}c_{01}+7c_{10}^2-16a_{02}+32b_{11}-32c_{20})a_{01}^2 \\
 - 2(2b_{01}-c_{10})(-6b_{01}b_{10}-b_{10}c_{10}+4a_{11})a_{01} \\
 + 4(2b_{01}-c_{10})^2(-b_{10}^2+a_{20}) + 4a_{01}^3c_{01}\bigr). \qedhere
\end{multline*}
\end{proof}
\begin{theorem}\label{theo_laa_nop}
With the conditions of Proposition~\ref{eqaaleh}, the configuration of affine asymptotic lines near to an affine cusp of Gauss point is topologically equivalent to one of the three models shown in Figures~\ref{fig1} and~\ref{fig3}.
\end{theorem}
\begin{proof}
This is a classical result and it can be found in~\cite{Dav1985,Davydov}. A sketch of the proof is the following. Let $p=\mfrac{dv}{du}$ and 
\begin{equation*}
F(u,v,p)=\left(-v+\lambda u^2+O(3)\right)+\left(1+\Lambda u^2+O(3)\right)p^2=0.
\end{equation*}
Consider the Lie-Cartan vector field $X(u,v,p)$ given as in~\eqref{eqLieCartanField}. We have that $X(0,0,0)=(0,0,0)$, thus $(0,0,0)$ is a singular point of $X$. The eigenvalues $\lambda_i$, $i=1,2$, of the linearization $DX(0,0)$ are given by
\begin{equation*}
\lambda_1=\frac{1}{2}\left(1+\sqrt{1-16\lambda}\right)\quad\text{and}\quad \lambda_2=\frac{1}{2}\left(1-\sqrt{1-16\lambda}\right).
\end{equation*}
Thus if $\lambda>\mfrac{1}{16}$, the projection of the integral curves of $X$ has a folded hyperbolic focus at the origin (Figures~\ref{fig1} and~\ref{fig3}, center). When $\lambda<\mfrac{1}{16}$, the product $\lambda_1\lambda_2=4\lambda$. If $\lambda<0$ we have a folded hyperbolic saddle and a folded hyperbolic node if $0<\lambda<\mfrac{1}{16}$. See Figures~\ref{fig1} and~\ref{fig3}, left a right respectively.
\end{proof}
%
%
\subsection{Affine asymptotic lines at a flat affine umbilic point}\label{subaacfumb}
In this part we are interested in the behavior of the affine asymptotic net in a neighborhood of a point where the two affine principal curvatures are equal to zero. In~\cite{BGC2020} are given the conditions required for the point $\alpha(0,0)=p\in S$ to be an affine umbilic point for $\alpha$ as in the Proposition \ref{prop_pick}.
\begin{lemma}[\cite{BGC2020}]\label{lem_umb_aff}
Let $p=\alpha(0,0)$ a non parabolic point of $S$. Then $p$ is an affine umbilic point if, and only if,  $q_{31}=-\varepsilon q_{13}$, and $q_{40}=q_{04}$.
\end{lemma}
\begin{definition}\label{def_fap}
A point $p$ at $S$ is called a flat affine umbilic point of $S$ if $p$ is simultaneously an affine umbilic and an affine parabolic point.
\end{definition}
Now we present the precise conditions for the point $p=\alpha(0,0)$ to be a flat affine umbilic point.
\begin{lemma}\label{lemm_faup}
Under the conditions of Lemma~\ref{lem_umb_aff}, $p=\alpha(0,0)$ is a flat affine umbilic point if, and only if, $q_{31}=-\varepsilon q_{13}$, $q_{40}=q_{04}$, and $q_{22}=-\varepsilon\left(-2\sigma^2+q_{40}\right)$.
\end{lemma}
\begin{proof}
By direct verification.
\end{proof}
Affine umbilic points are (some of the) singularities of affine principal nets. Since at a flat affine umbilic point $p$ we have that $l(p)=m(p)=n(p)=0$, then $k^{\mathrm{aff}}_{n}(w)=0$ for all $w\in T_pS$. Therefore, it follows that flat affine umbilic points are singularities of the affine asymptotic net.
\begin{proposition}\label{prop_edb_mors_fornor}
Let $p=\alpha(0,0)$ a flat affine umbilic point of $S$. Then by a change of coordinates the binary differential equation of the affine asymptotic lines can be written as
\begin{equation}\label{eq_edb_mors_fornor}
 \left(-\varepsilon_1 v+O(2)\right)du^2+\left(-2\varepsilon_1 u+O(2)\right)dudv+\left(v+O(2)\right)dv^2=0,
\end{equation}
where $\varepsilon_1=\pm 1$.
\end{proposition}
\begin{proof}
Note that we are in the conditions of the Lemma~\ref{lemm_faup}, then the BDE of the asymptotic lines is given by 
\begin{equation}\label{eq_prop_edb_mors_fornor}
 A(u,v)du^2+2B(u,v)dudv+C(u,v)dv^2=0,
\end{equation}
where
\begin{align*}
A(u,v) &= a_{10}u+a_{01}v + O(2), \\
B(u,v) &= a_{01}u+b_{01}v + O(2), \\
C(u,v) &= b_{01}u+c_{01}v + O(2), 
\end{align*}
with
\begin{align*}
a_{10} &= -6\varepsilon\sigma^3+6\varepsilon q_{40}\sigma-\varepsilon q_{50}-q_{32}, \\
a_{01} &= -\varepsilon q_{41}-2q_{13}\sigma-q_{23}, \\
b_{01} &= -6\sigma^3+2q_{40}\sigma-\varepsilon q_{32}-q_{14}, \\
c_{01} &= -6\varepsilon q_{13}\sigma-\varepsilon q_{23}-q_{05}.
\end{align*}
We apply a local diffeomorphism at $(0,0)$, $\mu\colon (x,y)\mapsto(u(x,y),v(x,y))$ where
\begin{equation}\label{eq_diff_loc_morse}
u(x,y)=s_{10}x+s_{01}y+O(2)\quad\text{and}\quad v(x,y)=t_{10}x+t_{01}y+O(2),
\end{equation}
with $s_{10}t_{01}-s_{01}t_{10}\neq0$. By a suitable choice of the coefficients of $\mu$ we obtain that the BDE takes the form given in~\eqref{eq_edb_mors_fornor}.
\end{proof}
\begin{proposition}\label{prop_aal_morse}
The configuration of the affine asymptotic lines near at a flat affine umbilic point is topologically equivalent to one of the two models shown in Figure~\ref{fig6}.
\begin{figure}[htb!]
	\centering
	\includegraphics[width=.75\textwidth,clip]{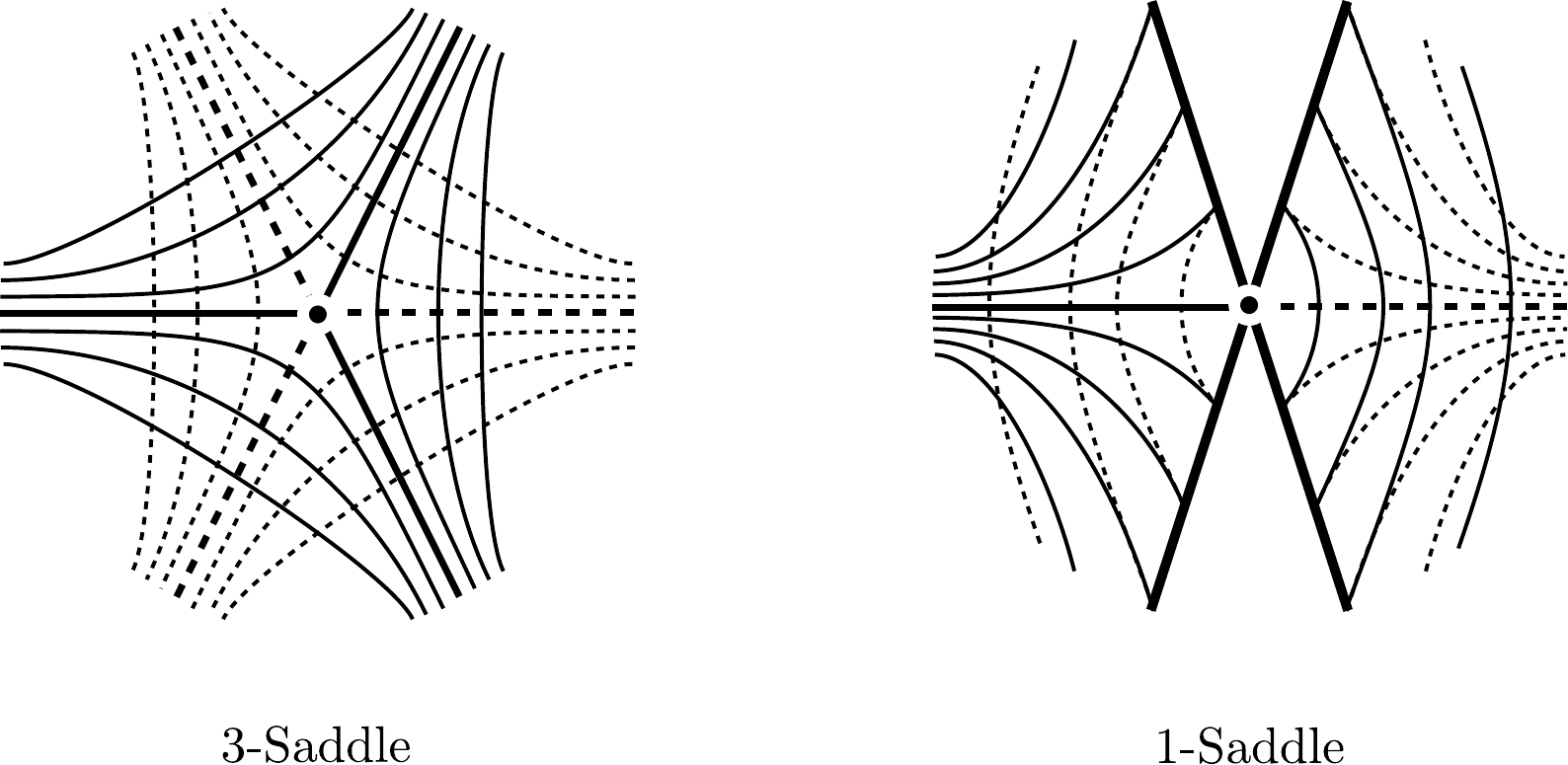}
	\caption{\small Affine asymptotic nets in a neighborhood of a flat affine umbilic point.}
	\label{fig6}
\end{figure}
\end{proposition}
\begin{proof}
Note that the discriminant function $\delta\colon\mathbb{R}^2\rightarrow\mathbb{R}$ of the BDE~\eqref{eq_edb_mors_fornor}, given by $\delta(u,v)=4(x^2+\varepsilon_1y^2)+O(3)$ satisfies that the discriminant set of the BDE~\eqref{eq_edb_mors_fornor} coincides with the zero level of $\delta$, this is with $\delta^{-1}(0)$. The function $\delta$ has a Morse singularity at $(0,0)$ and locally, when $\varepsilon_1=1$ the discriminant set is an isolated point $(0,0)$ and the asymptotic net is defined in a neighborhood of $(0,0)$. When $\varepsilon_1=-1$ the discriminant set is given by two smooth curves that cross each other transversally at $(0,0)$ and this determines four regions in the plane where only in two to them the integral curves of~\eqref{eq_edb_mors_fornor} are defined (see Figure~\ref{fig7}).
\begin{figure}[htb!]
	\centering
	\includegraphics[width=.55\textwidth,clip]{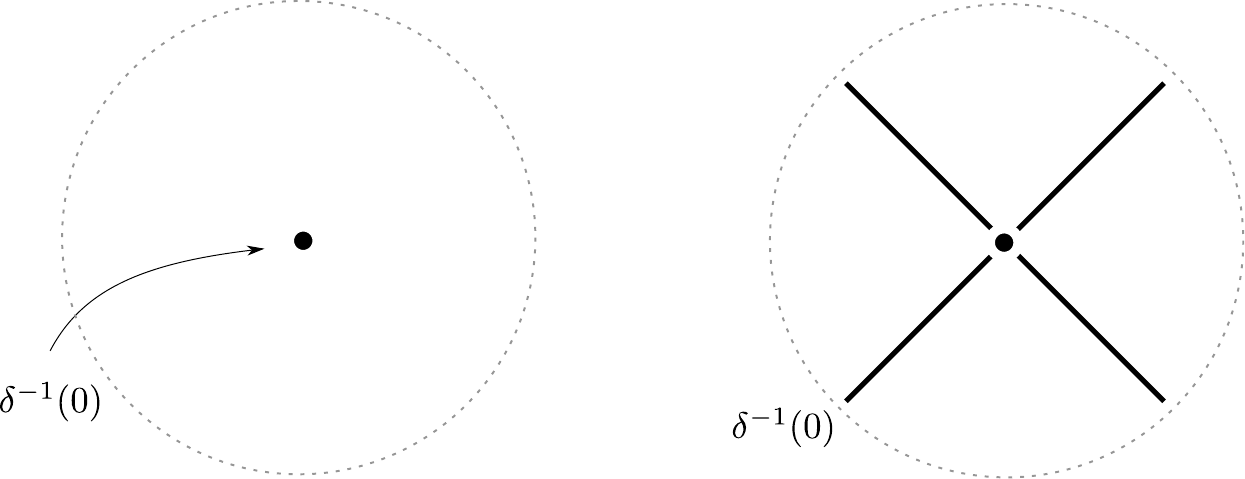}
	\caption{\small The discriminant set in a neighborhood of $(0,0)$. At the left hand side the set $\delta^{-1}(0)$ is locally an isolated point. At the right hand side the set $\delta^{-1}(0)$ induces a stratification formed by four open regions with $\delta>0$ in two of them.}
	\label{fig7}
\end{figure}
Let $p=\mfrac{dy}{dx}$ and consider the function $F(x,y,p)=-\varepsilon_1y-2\varepsilon_1up+vp^2=0$, and the Lie-Cartan vector field $X(u,v,p)$ given as in~\eqref{eqLieCartanField} again. When $u=v=0$, the vector field $X(0,0,p)=(0,0,p(p^2-3\varepsilon))$, thus the singularities of $X$ are given by $p(p^2-3\varepsilon_1)=0$. If $\varepsilon_1=-1$ then $X$ has only one singularity at $(0,0,0)$ whose eigenvalues of the projection of their linearization are $2$ and $-3$, this is one topological saddle (see Figure~\ref{fig6}, right). When $\varepsilon_1=1$ the vector field $X$ has three singularities at $(0,0,p_i)$, where $p_i\in\left\{-\sqrt{3},0,\sqrt{3}\right\}$ and it is easy to see that all the singularities are topological saddles (see Figure~\ref{fig6}, left). This completes the proof.
\end{proof}
\begin{remark}
In~\cite{Bruce1995} a Euclidean version of Proposition~\ref{prop_edb_mors_fornor} is given. More precisely, in the Euclidean case we have the same topological behavior of asymptotic lines near to a flat umbilic point.
\end{remark}
%
%
%

%
%


\section{Blaschke's asymptotic lines in a neighborhood of the Euclidean parabolic set}\label{sec:section5}
In this section we consider the behavior of the affine asymptotic nets near to the parabolic set. The parabolic set is generically formed by smooth curves (see~\cite{Banchoff,Bleecker1978,Izumiya2016}). Here we will consider the case when the parabolic set is a regular curve. As before, let us consider a smooth surface $S$ parametrized in a Monge chart $\alpha(u,v)=(u,v,h(u,v))$. In~\cite{BGC2020} it is shown that in this chart the expressions for $l$, $m$ and $n$ are
\begin{align*}
l &=  -\frac{1}{16}\frac{1}{(h_{uu}h_{vv}-h_{uv}^2)^{2}} \Bigl(-4(h_{vv}h_{uuuu}-2h_{uv}h_{uuuv})(h_{uu}h_{vv}-h_{uv}^2) \\
 &\qquad - 4h_{uu}(h_{uu}h_{vv}-h_{uv}^2)h_{uuvv} + 7h_{vv}^2h_{uuu}^2+3h_{uu}^2h_{uvv}^2 \\
 &\qquad + \bigl(-28h_{uuv}h_{uv}h_{vv}+2(h_{uu}h_{vv}+8h_{uv}^2)h_{uvv}-4h_{vvv}h_{uu}h_{uv}\bigr)h_{uuu} \\
 &\qquad + 12(h_{uu}h_{vv}+h_{uv}^2)h_{uuv}^2 + 4(h_{uu}^2h_{vvv}-6h_{uu}h_{uv}h_{uvv})h_{uuv}\Bigr),
\end{align*}
\begin{align*}
m &= -\frac{1}{16}\frac{1}{(h_{uu}h_{vv}-h_{uv}^2)^{2}} \Bigl(-4(h_{vv}h_{uuuv}-2h_{uv}h_{uuvv})(h_{uu}h_{vv}-h_{uv}^2) \\
 &\qquad + (7h_{vv}^2h_{uuv}-10h_{uv}h_{vv}h_{uvv}+(-h_{uu}h_{vv}+4h_{uv}^2)h_{vvv})h_{uuu} \\
 &\qquad - 4\left(h_{uu}h_{vv}-h_{uv}^2\right)h_{uvvv}-18h_{uuv}^2h_{uv}h_{vv}+7h_{uvv}h_{vvv}h_{uu}^2 \\
 &\qquad + \bigl((15h_{uu}h_{vv}+24h_{uv}^2)h_{uvv}-10h_{uu}h_{uv}h_{vvv}\bigr)h_{uuv}-18h_{uvv}^2h_{uu}h_{uv}\Bigr),
\end{align*}
\begin{align*}
n &= -\frac{1}{16}\frac{1}{(h_{uu}h_{vv}-h_{uv}^2)^{2}} \Bigl(-4(h_{vv}h_{uuvv}-2h_{uv}h_{uvvv})(h_{uu}h_{vv}-h_{uv}^2)\\
 &\qquad - 4h_{uu}h_{vvvv}(h_{uu}h_{vv}-h_{uv}^2)+4(-h_{uv}h_{vv}h_{vvv}+h_{uvv}h_{vv}^2)h_{uuu}\\
 &\qquad + 3h_{uuv}^2h_{vv}^2+2\bigl(-12h_{uv}h_{vv}h_{uvv}+(h_{uu}h_{vv}+8h_{uv}^2)h_{vvv}\bigr)h_{uuv}\\
 &\qquad + 12(h_{uu}h_{vv}+h_{uv}^2)h_{uvv}^2-28h_{uvv}h_{vvv}h_{uu}h_{uv}+7h_{vvv}^2h_{uu}^2\Bigr).
\end{align*}
Note that the BDE~\eqref{eqlaa} is a homogeneous equation, thus it can be extended to the parabolic set, defined by $h_{uu}h_{vv}-h_{uv}^2=0$ through the following equation
\begin{equation}\label{eqlaaps}
A(u,v)du^2+2B(u,v)dudv+C(u,v)dv^2=0
\end{equation}
where $A,B,C$ are the numerators of $l,m,n$ respectively. As the singular set of the affine asymptotic net is exactly the affine parabolic set, we need known how is the interaction between them. 
In~\cite{Davis2008} the author studies some relation between this sets and the next result was obtained.
\begin{proposition}[\cite{Davis2008}]\label{prop_davis_parsets}
Consider a point which is both a parabolic point and an affine parabolic point. We have
\begin{enumerate}
\item If the parabolic point is an ordinary parabolic point or a non-degenerate cusp of Gauss then the affine parabolic set will be non-singular.
\item If the parabolic point is an ordinary parabolic point then the parabolic set and the affine parabolic set will meet transversally.
\item If the parabolic point is an cusp of Gauss, then the parabolic set and the affine parabolic set will be tangent.
\end{enumerate}
\end{proposition}
\begin{proof}
Can be found in~\cite[Chapter 9]{Davis2008}.
\end{proof}
\begin{remark}\label{rem_par_sets}
The condition required in Proposition~\ref{prop_davis_parsets} that a point should be both parabolic and affine parabolic point is very restrictive. In general, if a point $p\in S$ is parabolic it does not have to be affine parabolic. This can be verified by direct calculations.
\end{remark}
By Remark~\ref{rem_par_sets} and Proposition~\ref{prop_davis_parsets}~\itm{iii} we conclude that the next step in our analysis should be to consider the neighborhood of a cusp of Gauss point.
%
%
\subsection{Affine asymptotic lines near to a cusp of Gauss point}
We will adopt the normal form~\eqref{eq_para_parb_cg} below for the required analysis 
\begin{align}\label{eq_para_parb_cg}
 h(u,v) & = v^2+q_{21}u^2v+q_{03}v^3+q_{40}u^4+q_{31}u^3v \notag \\
 &\qquad + q_{22}u^2v^2+q_{13}uv^3+q_{04}v^4+q_{50}u^5+q_{41}u^4v+q_{32}u^3v^2 \\
 &\qquad + q_{23}u^2v^3+ q_{14}uv^4+q_{05}v^5+q_{60}u^6+ q_{51}u^5v+q_{42}u^4v^2 \notag \\
 &\qquad + q_{33}u^3v^3+q_{24}u^2v^4+q_{15}uv^5+q_{06}v^6+O(7), \notag
\end{align}
with $q_{21}^2-4q_{40}\neq0$. The reason is that under these conditions $\alpha(0,0)$ is a cusp of Gauss point.
\begin{proposition}
Let $p=\alpha(0,0)$ a cusp of Gauss point, then $p$ is also a affine cusp of Gauss point.
\end{proposition}
\begin{proof}
Evaluating the expression~\eqref{eq_para_parb_cg} in~\eqref{eqlaaps} at $(u,v)=(0,0)$ we obtain
\begin{equation*}
A(0,0)du^2+2B(0,0)dudv+C(0,0)dv^2=-48q_{21}^2dv^2=0
\end{equation*}
and the result follows.
\end{proof}
Now, the principal result of this section can be stated as follows.
\begin{theorem}
Let $\alpha\colon \mathbb{R}^2\rightarrow\mathbb{R}^3$ be a local parametrization of the surface $S$ in a small neighborhood of a cusp of Gauss point. Then the configuration of the affine asymptotic lines is locally topologically equivalent to 
\begin{equation}\label{eq_th_cg}
\Bigl(v-\frac{5}{4}\,u^2+O(3)\Bigr)du^2 + \bigl(1+O(3)\bigr)dv^2=0.
\end{equation}
In the source, the Euclidean and affine parabolic sets have contact of order two at the origin and the local behavior of the affine asymptotic net together with the parabolic set is shown in Figure~\ref{fig8}.
\begin{figure}[htb!]
  \centering
  \includegraphics[width=.5\textwidth,clip]{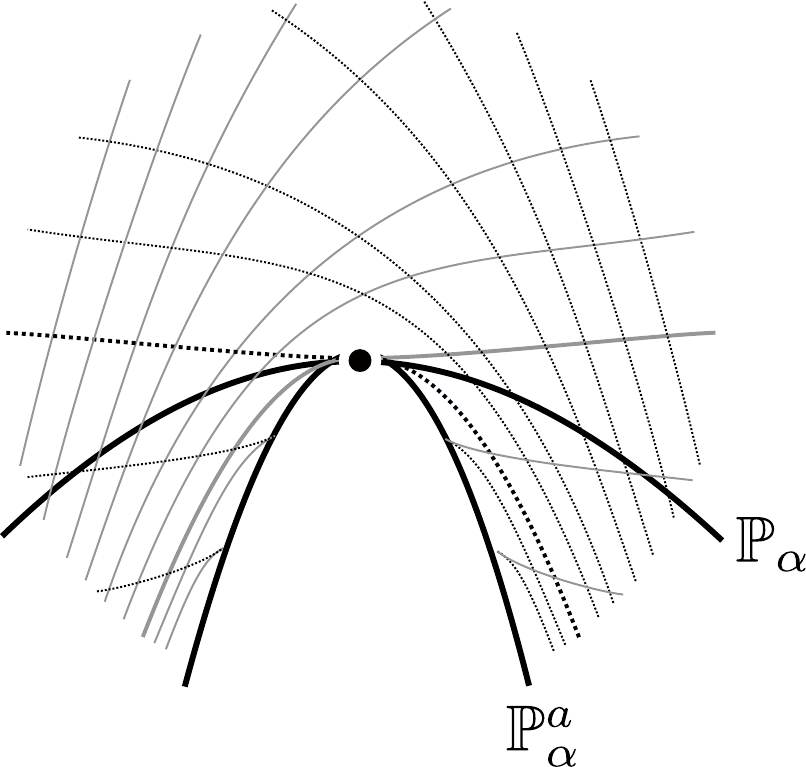}
  \caption{\small Affine asymptotic lines near a cusp of Gauss point.}
  \label{fig8}
\end{figure}
\end{theorem}
\begin{proof}
Evaluating~\eqref{eq_para_parb_cg} in the BDE~\eqref{eqlaaps} we obtain
\begin{align*}
A(u,v) & = -192\,q_{21}(q_{21}^2-4q_{40})v - 192\,(q_{21}^4-18q_{21}^2q_{40}+60q_{40}^2)u^2 \\
 &\qquad -384\,(2q_{21}^2q_{31}-10q_{21}q_{50}+15q_{31}q_{40})uv-16\,(48q_{03}q_{21}^3  \\
 &\qquad -288\,q_{03}q_{21}q_{40}+52q_{21}^2q_{22}-48q_{21}q_{41}-48q_{22}q_{40}+63q_{31}^2) v^2 \\
 &\qquad + O(3),
\end{align*}
\begin{align*}
  B(u,v) &= 96q_{21}(3q_{21}^2-14q_{40})u-144q_{31}q_{21}v+48(41q_{21}^2q_{31}-70q_{21}q_{50} \\
 &\qquad- 60\,q_{31}q_{40})u^2+16(84q_{03}q_{21}^3-468q_{03}q_{21}q_{40}+62q_{21}^2q_{22}-36q_{21}q_{41} \\
 &\qquad + 168\,q_{22}q_{40}-27q_{31}^2)uv-48(15q_{03}q_{21}q_{31}+2q_{13}q_{21}^2-q_{21}q_{32} \\
 &\qquad + 10q_{22}q_{31})v^2 + O(3),
\end{align*}
\begin{align*}
 C(u,v) &= -48\,q_{21}^2-288q_{21}q_{31}u-64q_{21}(6q_{03}q_{21}+q_{22})v-16(48q_{03}q_{21}^3 \\
 &\qquad - 108q_{03}q_{21}q_{40}-34q_{21}^2q_{22}+36q_{21}q_{41}+48q_{22}q_{40}+27q_{31}^2)u^2 \\
 &\qquad - 192\,(9q_{03}q_{21}q_{31}-2q_{13}q_{21}^2+12q_{13}q_{40}+q_{21}q_{32}+3q_{22}q_{31})uv \\
 &\qquad - 32\,(54q_{03}^2q_{21}^2+21q_{03}q_{21}q_{22}+6q_{04}q_{21}^2+18q_{13}q_{31}-3q_{21}q_{23} \\
 &\qquad + 2q_{22}^2)v^2 + O(3).
\end{align*}
As before, let us consider the change of coordinates $\Psi\colon\mathbb{R}^2\rightarrow\mathbb{R}^2$ given by the map $\Psi(x,y)=\left(u(x,y),v(x,y)\right)$, where
\begin{align*}
u(x,y)&=s_{10}x+s_{01}y+s_{20}x^2+s_{11}xy+s_{02}y^2+s_{30}x^3+s_{21}x^2y+s_{12}xy^2+s_{03}y^3+\ldots, \\
v(x,y)&=t_{01}y+t_{20}x^2+t_{11}xy+t_{02}y^2+t_{30}x^3+t_{21}x^2y+t_{12}xy^2+t_{03}y^3+\ldots,
\end{align*}
and $s_{10}t_{01}\neq0$. By a procedure analogous to that used in the proof of Proposition~\ref{eqaaleh} and making appropriate choices in the coefficient of $\Psi$, we obtain the desired form for the BDE of the affine asymptotic lines. In the source the Euclidean and affine parabolic sets are given (ommiting higher order terms h.o.t.) respectively by
\begin{equation*}
v=\frac{q_{21}^2-6q_{40}}{q21}u^2+\text{ h.o.t.}\quad\text{and}\quad v=2\frac{4q_{21}^2-17q_{40}}{q_{21}}u^2+\text{ h.o.t.},
\end{equation*}
thus they have contact of order two at the origin. For the final part, let $(u,\eta(u))$, with $\eta(u)=\mfrac{q_{21}^2-6q_{40}}{q21}u^2+\text{\,h.o.t.}$, be a local parametrization of the parabolic set. Evaluating in the discriminant function of the BDE~\eqref{eqlaaps} we obtain
\begin{equation*}
\left(B^2-AC\right)\left(u,\eta(u)\right)=16128q_{21}^2\left(q_{21}^2-4q_{40}\right)^2u^2+ \text{ h.o.t.}
\end{equation*}
Near to the cusp of Gauss point and along the parabolic set the discriminant function satisfies $B^2-AC>0$, and this concludes the proof. 
\end{proof}
\subsection{Affine asymptotic lines near to a flat umbilic point}
Finally we will consider a small neighborhood of a flat umbilic point. It is worth to point out that there are no flat umbilic points for general surfaces. In other words, flat umbilic points disappear under small perturbations, and thus, being a flat umbilic point is not a generic condition.

It is shown in~\cite{Bruce1995} that the Euclidean asymptotic net near to a flat umbilic point has one of the local topological models appearing in Figure~\ref{fig6}. For the affine case the analogous result is given by Proposition~\ref{prop_aal_morse}.

It is well known that a local chart for a neighborhood of a flat umbilic point is given by $\alpha(u,v)=\left(u,v,h(u,v)\right)$, where
\begin{align}\label{eq_class_up}
h(u,v) &= q_{30}u^3+q_{21}u^2v+q_{12}uv^2+q_{03}v^3+q_{40}u^4+q_{31}u^3v+q_{22}u^2v^2 \\
 &\qquad + q_{13}uv^3+q_{04}v^4+ q_{50}u^5+q_{41}u^4v+q_{32}u^3v^2+q_{23}u^2v^3 \notag \\ 
 &\qquad + q_{14}uv^4 + q_{05}v^5+O(6). \notag
\end{align}
For simplicity, in calculations we will use the following local parametrization 
\begin{proposition}[Special chart]\label{prop_for-norm-fup}
Let $S$ a smooth surface locally parametrized by $\alpha\left(u,v\right)=\left(u,v,h\left(u,v\right)\right)$ and let $p=X(0,0)$. Assume that $p$ is a flat umbilic point, then $h\left(u,v\right)$ can be written as
\begin{align}\label{for-norm-fup}
h(u,v) &= u^3+3\varepsilon uv^2+q_{40}u^4+q_{31}u^3v+q_{22}u^2v^2+q_{13}uv^3+q_{04}v^4  \\
 &\qquad + q_{50}u^5+q_{41}u^4v+q_{32}u^3v^2+q_{23}u^2v^3+q_{14}uv^4+q_{05}v^5+O(6), \notag
\end{align}
where $\varepsilon=\pm1$.
\end{proposition}
\begin{proof}
Consider the change of coordinates $\Theta_1\colon (x,y)\mapsto(u(x,y),v(x,y))$ given by 
\begin{equation*}
u(x,y)=x+a_{01}y\quad\text{and}\quad v(x,y)=y+b_{10}x,\quad\text{with}\quad 1-a_{01}b_{10}\neq0.
\end{equation*}
Applying it to $\alpha$ with $h$ as in~\eqref{eq_class_up} we obtain
\begin{equation*}
\alpha\circ\Theta_1(x,y)=\left(x+a_{01}y,y+b_{10}x,h\circ\Theta(x,y)\right).
\end{equation*}
Taking now the linear transformation
\begin{equation*}
T_1\colon (U,V,W)\longmapsto\left(r^{-1}\left(U-a_{01}V\right),r^{-1}\left(V-b_{10}U\right),r\,W\right)
\end{equation*}
where $r=1-a_{01}b_{10}$, we obtain 
\begin{equation*}
T_1\circ\alpha\circ\Theta_1(x,y)=\left(x,y,\Sigma(x,y)\right),
\end{equation*}
where
\begin{equation*}
\Sigma(x,y)=Q_{30}x^3+Q_{21}x^2y+Q_{12}xy^2+Q_{03}y^3+O(4).
\end{equation*}
Note that the matrix associated with $T_1$ has determinant equal to one. Let $\alpha_1(x,y)=T_1\circ\alpha\circ\Theta_1(x,y)$. It is clear that the $Q_{ij}$'s depend on the coefficients of the 3-jet of~\eqref{eq_class_up} and the pair $a_{01},b_{10}$. Taking
\begin{equation*}
a_{01}=-\frac{3b_{10}^2q_{03}+2b_{10}q_{12}+q_{21}}{b_{10}^2q_{12}+2b_{10}q_{21}+3q_{30}}
\end{equation*}
we have that $Q_{21}$ vanishes. $Q_{03}$ disappears if $b_{01}$ is a real root of the equation
\begin{equation*}
R_3b_{01}^3+3R_2b_{01}^2+3R_1b_{01}+R_0=0,
\end{equation*}
where
%
\begin{itemize}
	\item $R_3=-27q_{03}^2q_{30}+9q_{03}q_{12}q_{21}-2q_{12}^3$,
	\item $R_2=3\left(-9q_{03}q_{12}q_{30}+6q_{03}q_{21}^2-q_{12}^2q_{21}\right)$,
	\item $R_1=3\left(9q_{03}q_{21}q_{30}-6q_{12}^2q_{30}+q_{12}q_{21}^2\right)$,
	\item $R_0=-9q_{12}q_{21}q_{30}+2q_{21}^3$.
\end{itemize}
%
It follows that $\Sigma(x,y)=Q_{30}x^3+Q_{12}xy^2+O(4)$. Consider now the change of coordinates $\Theta_2\colon (u,v)\mapsto(x(u,v),y(u,v))$ given by 
\begin{equation*}
x(u,v)=a\,u\quad\text{and}\quad y(u,v)=b\,v,\quad\text{with}\quad a,b\neq0,
\end{equation*}
and the linear transformation
\begin{equation*}
T_2\colon (U,V,W)\longmapsto\left(\frac{1}{a}\,U,\frac{1}{b}\,V,a\,b\,W\right).
\end{equation*}
It follows that taking $b=\mfrac{1}{a^4q_{30}}$ and after $a=\sqrt[^{10}]{\Bigl|\mfrac{q_{12}}{q_{30}^3}\Bigr|}$ we obtain
\begin{equation*}
T_2\circ\alpha\circ\Theta_2(u,v)=\left(u,v,u^3+\varepsilon uv^2+O(4)\right),
\end{equation*}
where $\varepsilon=\pm1$ just as we wanted to show.
\end{proof}
\begin{lemma}\label{lemma_epsilonig1}
Let S a smooth surface $S$ locally parametrized by $\alpha(u,v)=(u,v,h(u,v))$ as in the Proposition~\ref{prop_for-norm-fup}. Then if $\varepsilon=1$, in a small neighborhood of $p=\alpha(0,0)$ do not exist affine asymptotic lines. 
\end{lemma}
\begin{proof}
In this chart, the discriminant function $\delta$, associated with the BDE~\eqref{eqlaaps} satisfies
\begin{equation}\label{eq_dis_epsilon}
\delta(u,v) = -589824\,\varepsilon (\varepsilon v^2-3u^2)^2+O(5).
\end{equation}
It follows that near to $p=\alpha(0,0)$, $\delta(u,v)\leq 0$ and this finishes the proof.
\end{proof}
Regarding Lemma~\ref{lemma_epsilonig1}, it remains to analyze the case when $\varepsilon=-1$ in the parametrization given in Proposition~\ref{prop_for-norm-fup}. 
\begin{theorem}\label{theoumbplaff}
Under the conditions of Lemma~\ref{lemma_epsilonig1} take $\varepsilon=-1$. Then the discriminant set $\delta^{-1}(0)$ of the BDE~\eqref{eqlaaps} is locally an isolated point $(0,0)$ and the local behavior of the affine asymptotic curves near the origin is as illustrated in Figure~\ref{fig9}.
\begin{figure}[htb!]
	\centering
	\includegraphics[width=.4\textwidth,clip]{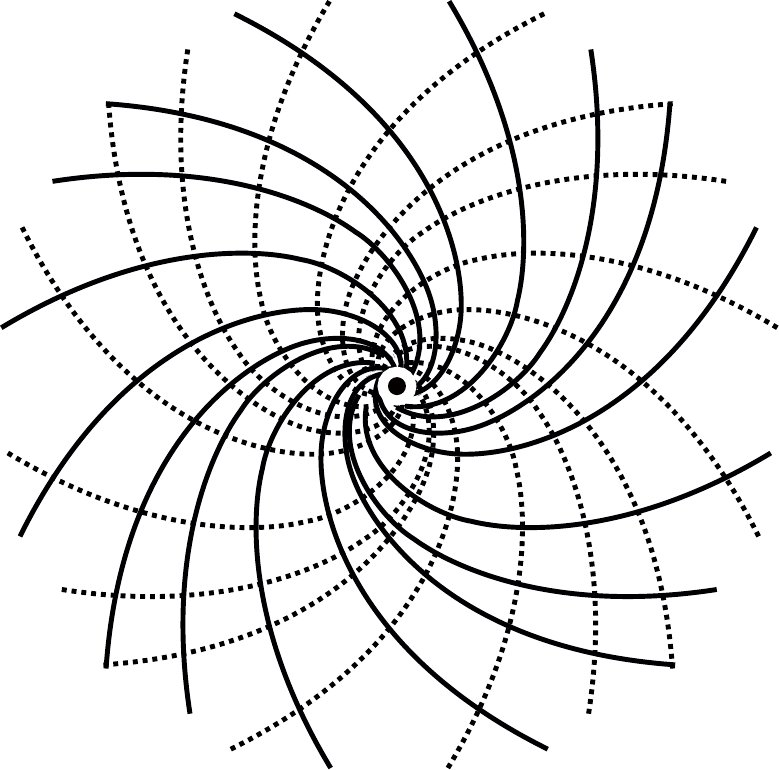}
	\caption{\small Affine asymptotic lines near a flat umbilic point.}
	\label{fig9}
\end{figure}
\end{theorem}
\begin{proof}
The first part is a direct consequence of taking $\varepsilon=-1$ in~\eqref{eq_dis_epsilon}. For the second part consider the polar blowing-up for the BDE~\eqref{eqlaaps} given by
\begin{equation*}
\varphi(r,t)=\left(r\,\cos(t),r\,\sin(t)\right),\quad\text{$r\geq0$ and $0<t<2\,\pi$}.
\end{equation*}
The new BDE in the variables $r,t$, after dividing by $r^2$, is given by
\begin{equation}\label{edb_blow_polar}
\widebar{A}(r,t)\,dr^2+2r\widebar{B}(r,t)\,drdt+r^2\widebar{C}(r,t)\,dr^2,
\end{equation}
where
\begin{align*}
	\widebar{A}(r,t) & = (1+2\cos^2(t))^2+\Lambda_{1}(r,t),\\
	\widebar{B}(r,t) & = -2\sin(t)\cos(t)(1+2\cos^2(t))+\Lambda_{2}(r,t),\\ 
	\widebar{C}(r,t) & = -4\left(\cos^4(t)+\sin^2(t)\right)+\Lambda_{3}(r,t).
\end{align*}
with $\Lambda_{i}(0,t)=0$ for $i=1,2,3$. Note that $\widebar{A}(0,t)=(1+2\cos^2(t))^2\neq 0$, and then the BDE~\eqref{edb_blow_polar} has not singularities. Consider now the vector fields $Y_i$ defined by kernel of the differential forms 
\begin{equation}
\eta_i(r,t)=r\Bigl(-\widebar{B}(r,t)+(-1)^i\sqrt{\smash[b]{(\widebar{B}^2-\widebar{A}\widebar{C})(r,t)}}\Bigr)\,dr + \widebar{A}(r,t)\,dt,\quad\ i=1,2.
\end{equation}
Near $r=0$, the vector fields $Y_i$ span the line fields associated with the BDE~\eqref{edb_blow_polar} and they do not have singularities either. Considering $Y_i=(\dot{r},\dot{t})_i=\bigl(f^i_1(r,t),f^i_2(r,t)\bigr)$ we can verify that near $r=0$ the relations $f^1_1(r,t)<0$ and $f^1_2(r,t)>0$ hold for $Y_1$, while $f^2_1(r,t)>0$ and $f^2_2(r,t)>0$ hold for $Y_2$. Their phase portraits are shown in Figure~\ref{fig10}.
\begin{figure}[h!]
	\centering
	\includegraphics[width=.6\textwidth,clip]{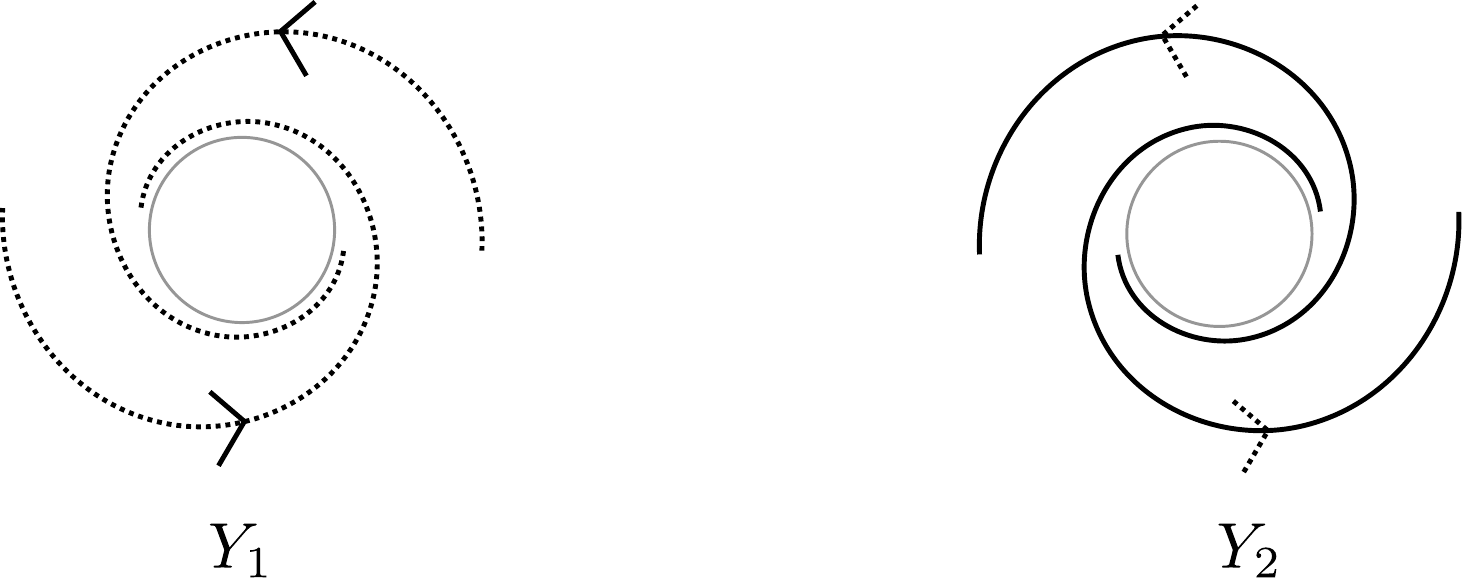}
	\caption{\small Local phase portraits of the vector fields $Y_1$ and $Y_2$, left and right respectively.}
	\label{fig10}
\end{figure}

The blowing-down $\varphi^*$ of the integral curves of the BDE~\eqref{edb_blow_polar} appears in Figure~\ref{fig11}.
\begin{figure}[h!]
	\centering
	\includegraphics[width=.7\textwidth,clip]{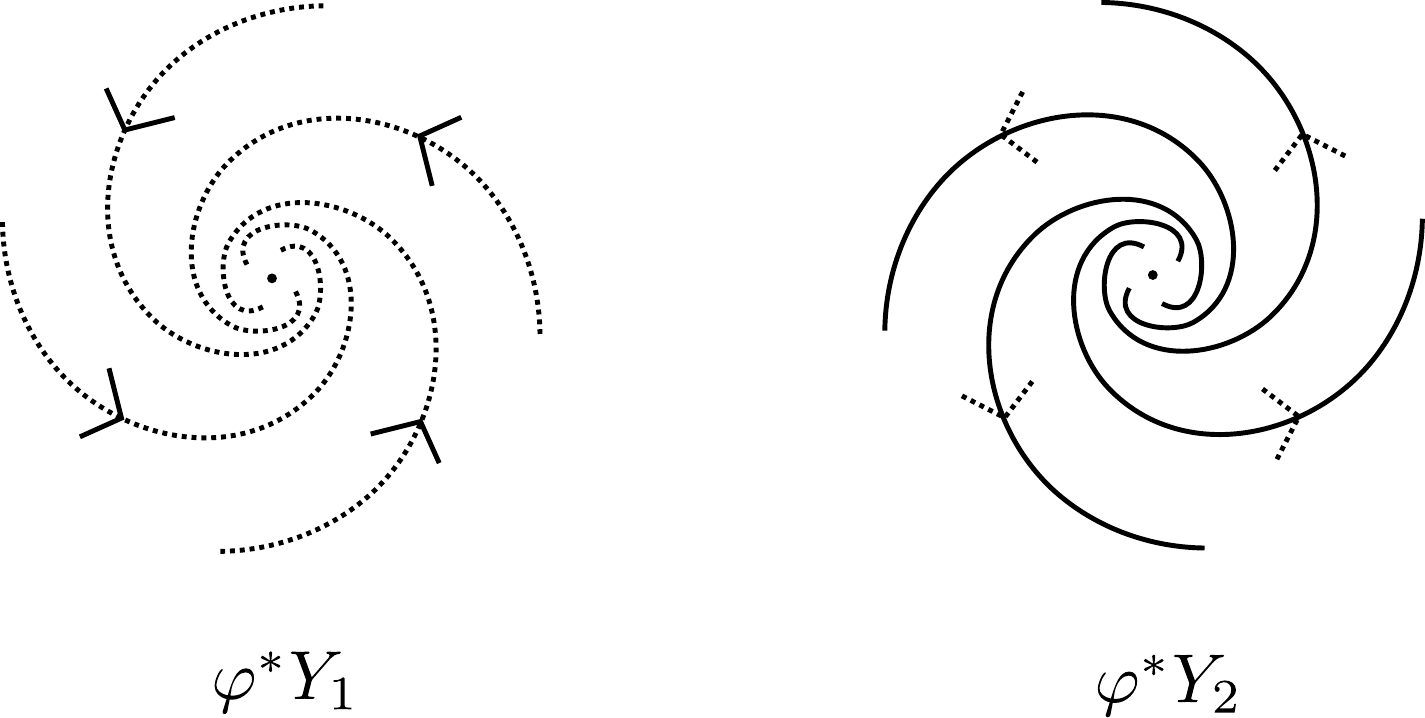}
	\caption{\small Configurations of the integral curves via blowing-down of the solutions of the BDE \eqref{edb_blow_polar}.}
	\label{fig11}
\end{figure}

Thus, the vector fields $Y_i$, $i=1,2,$ have a topological focus at the origin that is stable, and unstable, respectively, and this completes the proof.
\end{proof}
\begin{remark}
In~\cite{Oliveira2000} the authors studied pairs of germs of differential 1-forms $(\alpha,\beta)$ in the plane where $\alpha,\beta$ are either regular or have singularity of type saddle, node or focus. Equivalently they considered the pairs of integral curves of direction fields in the plane defined by the kernels of the 1-forms. Fixing the 2-jet of the BDE~\eqref{eqlaaps} in the chart given in the Proposition~\ref{prop_for-norm-fup}, the Theorem~\ref{theoumbplaff} can be proved using the tools developed there. Curiously the picture in Figure~\ref{fig9} does not appear in that paper.
\end{remark}
%
%
\section{Examples of affine asymptotic lines}\label{sec:section6}

In this section it will be given an explicit example, showing the global behavior of affine asymptotic lines in the torus of revolution and its co-normal surface.
\begin{proposition}\label{prop:LA_toro}
	Consider the torus of revolution parametrized by   
	
	\[ \alpha(u,v)= \big( (R+r\cos u) \cos v, (R+r\cos u)\sin v, r\sin u\big),\quad\text{with $r<R$}. \]
	The differential equation of the affine asymptotic lines is given by:
	
	\begin{equation} \label{eq:LA_toro}
	\aligned 
	\mathcal{A} &= \big[ (15\cos^2u  -3)R^2+4\cos u\,(9\cos ^2u-2)r R+16\cos^4u \;r^2 \big]\,du^2\\
	&\qquad + \big [4\cos^2u \,(4r\cos^3 u\;  +3R\cos^2u+R)(R+r\cos u)\big]\,dv^2=0
	\endaligned
	\end{equation}
	Moreover the real solutions are defined in the rings $R_1=\{ 0<u_1<u<u_2<\pi \}$ and
	$R_2=\{ \pi<u_3<u<u_4<2\pi \}$ containing the parabolic set which is formed by the circles $u=\pi/2$ and $u=3\pi/2$.
	
\end{proposition}

\begin{proof} Straightforward calculations, using the definition of the affine normal $\xi$ in~\eqref{eq:na} and the formulas in~\eqref{AIffcoeff}, lead to:
\[
\aligned l &= \frac{(15\cos^2u-3)R^2+4\cos u\, (9\cos^2u-2)rR+16\cos^4u \; r^2}{16(R+r\cos u)^2\cos^2u},\\
m &= 0,\\
n &=  \frac{(3\cos^2u+1)R+4r\cos^3u}{4(R+r\cos u)}.
\endaligned
\]
With the hypothesis $R>r$ it follows that $l=0$, an equation of degree four in the variable $\cos u$ has two real solutions in the interval $(-1,1)$, and $n=0$ has no real solution in the variable $\cos u$ in the mentioned interval.
\end{proof}

\begin{proposition}\label{prop:LA_toro} The affine asymptotic lines of the torus of revolution are shown in  Figures~\ref{fig:LA_toro} and \ref{figTorus}.
	The parabolic set defined by  $u=\pi/2$ and $u=3\pi/2$ are   solutions of both families of affine asymptotic lines and in the neighborhood of the four circles defined by $u=u_i$ (i=1,2,3,4), boundary of the rings $R_1$ and $R_2$, the behavior of affine asymptotic lines is of cuspidal type.
 \begin{figure}[H]
   \begin{center}
       \def\svgwidth{1.0\textwidth}
       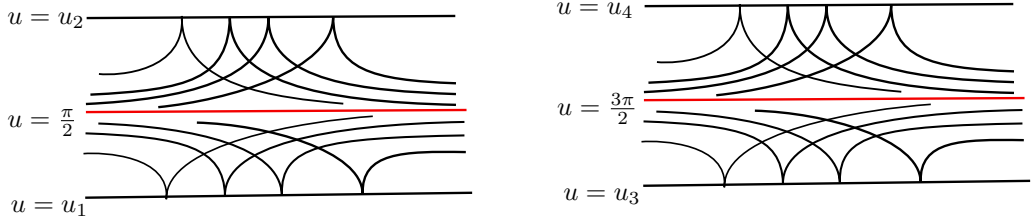
       \caption{Affine asymptotic lines of a torus of revolution (shown in the source).}
       \label{fig:LA_toro}
   \end{center}
\end{figure}
\begin{figure}[h!]
	\centering
	\includegraphics[width=.5\textwidth,clip]{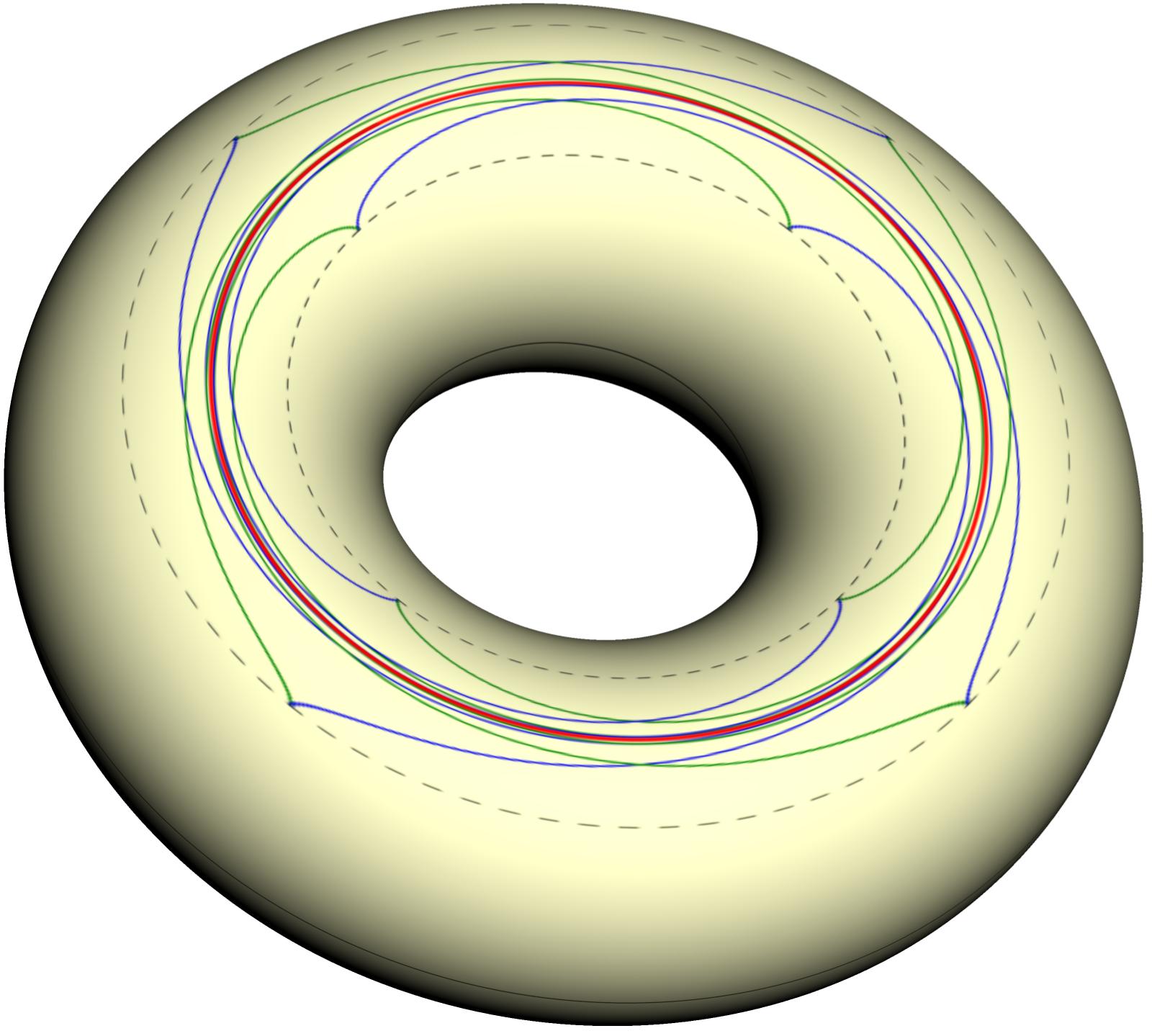}
	\caption{\small Affine asymptotic lines shown on the surface of the torus of revolution.}
	\label{figTorus}
\end{figure}
\end{proposition}
\begin{proof}
According to the method used in Section~\ref{sec:section5} for the analysis in a neighborhood of the parabolic set we consider the extended BDE which is given by $\bar{l}du^2+2\bar{m}dudv+\bar{n}dv^2=0$, where 
\begin{equation}
\aligned
\bar{l} &= (15\cos^2u-3)R^2+4rR\cos u\,(9\cos^2u-2) +16r^2 \cos^4u, \\
\bar{m} &= 0, \\
\bar{n} &= 4\cos^2u\,\big((3\cos^2u+1)R+4r\cos^3u \big)(R+r\cos u).
\endaligned
\label{eq:affine_torus}
\end{equation}
Along the parabolic set we have 
\begin{equation*}
\bar{l}\left(\frac{\pi}{2}\right)=\bar{l}\left(\frac{3\pi}{2}\right)=R\quad\text{and}\quad\bar{n}\left(\frac{\pi}{2}\right)=\bar{n}\left(\frac{3\pi}{2}\right)=0.
\end{equation*}
The tangent to the parabolic set is $(du,dv)=\left(0,1\right)$ and this part follows. Now, as $n=0$ has no real solution in the variable $\cos u$ as seen in the proof below, we will focus our attention in $l=0$ which corresponds to the affine parabolic set. Along the affine parabolic set $l=0$ and under these conditions  the unique affine asymptotic direction is given by $\left(1,0\right)$, which is transversal to the affine parabolic set and so the  behavior  is of cuspidal type as shown in Fig.\,\ref{fig:LA_toro}. This concludes the proof.
\end{proof}

The description of Euclidean asymptotic lines on the torus was  given in~\cite{White1907} and \cite{Garcia2009}.

\begin{proposition} The co-normal surface of the torus of revolution has two connected components with the following properties.

\begin{enumerate}
\item The parabolic set is formed of four circles as in Proposition~\ref{prop:LA_toro}.  
 
\item The hyperbolic region has four unbounded components.

\item The Euclidean asymptotic lines are as shown in Fig.\,\ref{fig:LA_toro}
\end{enumerate}
\end{proposition}
%
\begin{proof}
The co-normal surface $\nu$ is defined in the domain $(u,v)\in [0,2\pi)\times [0, 2\pi)$ with $u\neq \pm \frac{\pi}{2}$ and is given by
\begin{align*}
 \nu_{\varepsilon}&=\left[-\frac{ \cos u\cos v}{W}, -\frac{ \cos u\sin v}{W},-\frac{ \sin u }{W}\right]\\
 W &=  \left(\frac {\varepsilon\,\cos u }{r\, \left( R+r\,\cos u \right) }\right)^{\frac{1}{4}}
\end{align*}
Here $\varepsilon=1$ for $u\in (-\frac{\pi}{2}, \frac{\pi}{2})$ and $\varepsilon=-1$ for $u\in (\frac{\pi}{2}, \frac{3\pi}{2})$, see Fig.\,\ref{fig.geratrixconormal}.  
\begin{figure}[htpb!]
	\centering
	\includegraphics[width=.55\textwidth,clip]{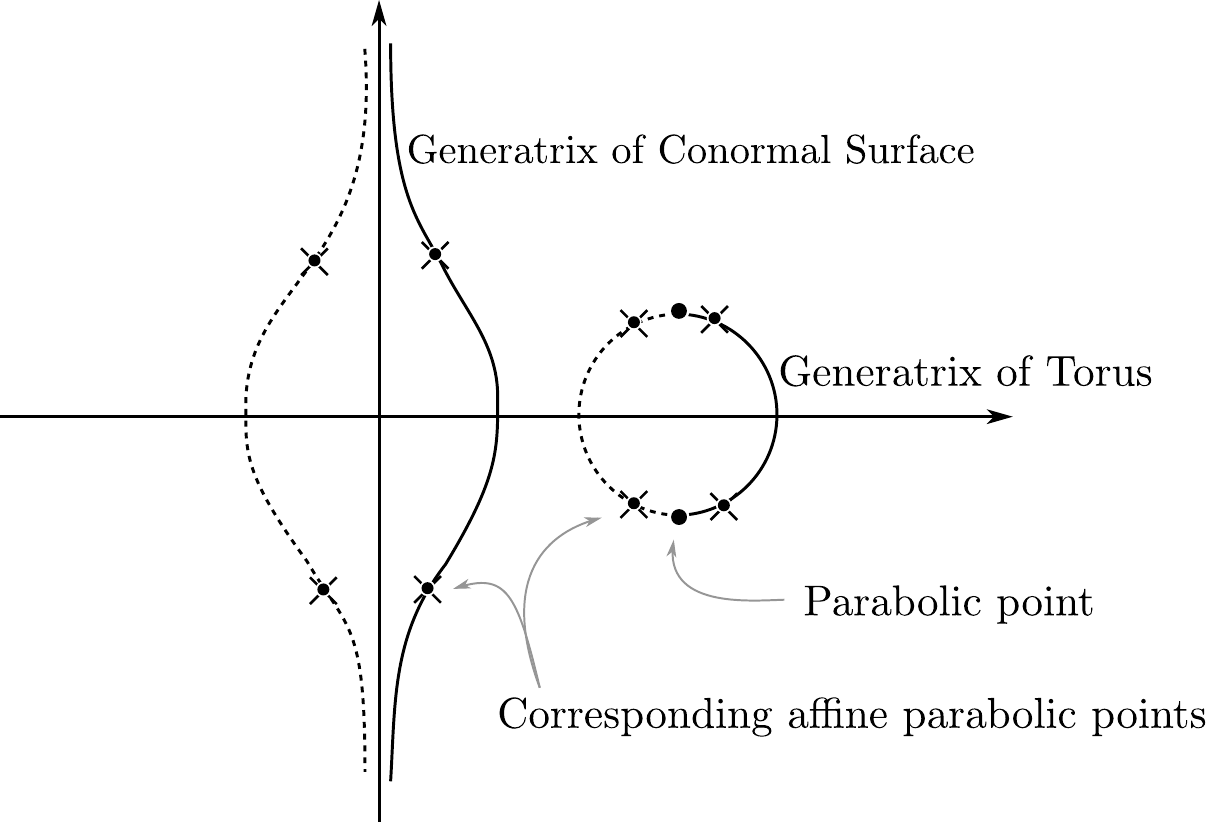}
	\caption{\small Corresponding generatrix curves of torus and the conormal surface.}
	\label{fig.geratrixconormal}
\end{figure}
\begin{figure}[htpb!]
	\centering
	\includegraphics[width=.85\textwidth,clip]{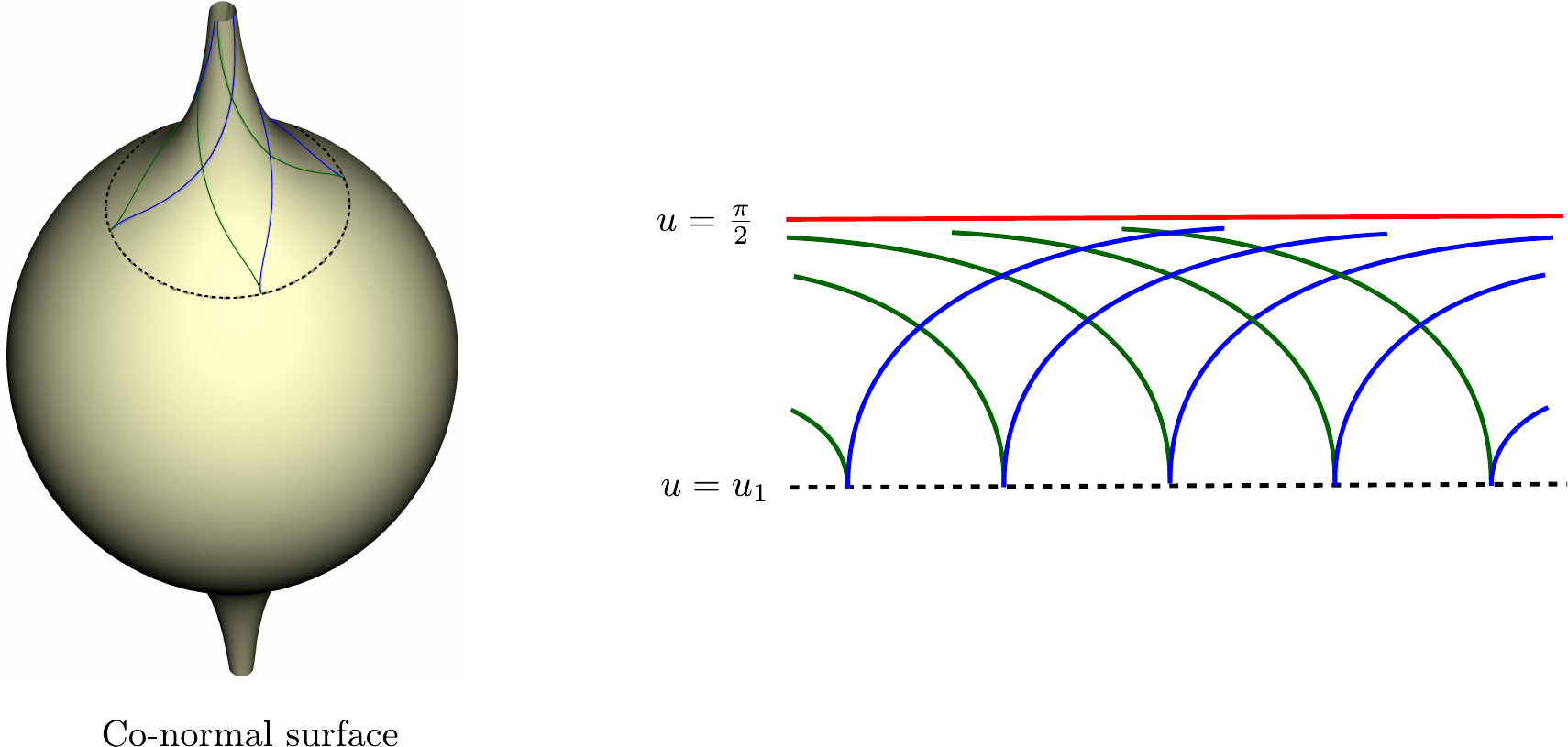}
	\caption{\small A connected component of the co-normal surface associated with the elliptical region of the torus (left).  Euclidean asymptotic curves of the co-normal surface of the torus (left) and at the source (right). When the asymptotic curve approaches the line $u=\frac{\pi}{2}$ in the source, the corresponding asymptotic curve on the co-normal surface goes away to infinity.}
	\label{fig15}
\end{figure}
The analysis of the asymptotic lines follows from Proposition~\ref{prop:LA_toro} and Theorem~\ref{th_rel_geo}. An illustration is presented in Fig.\,\ref{fig15}.
\end{proof}

%
%


\vspace{8mm}
\setlength{\parindent}{0pt}
\begin{minipage}{.58\linewidth}
\footnotesize
{\sc
Martín Barajas Sichac\'a \\
Departamento de Matem\'aticas \\
Pontificia Universidad Javeriana \\
Bogot\'a, Colombia \\
\textit{E-mail:} {\tt mabarajas@javeriana.edu.co} \\
}
\end{minipage}
\begin{minipage}{.58\linewidth}
\footnotesize
{\sc
Ronaldo A. Garcia \\
Instituto de Matem\'atica e Estat\'istica \\
Universidade Federal de Goi\'as \\
Goiânia-GO, Brazil \\
\textit{E-mail:} {\tt ragarcia@ufg.br} \\
}
\end{minipage}
\begin{minipage}{.42\linewidth}
\footnotesize
{\sc
Andr\'es Vargas \\
Departamento de Matem\'aticas \\
Pontificia Universidad Javeriana \\
Bogot\'a, Colombia. \\
\textit{E-mail:} {\tt a.vargasd@javeriana.edu.co} \\
}
\end{minipage}

\end{document}

%% file: LA_toro.pdf_tex
\begingroup%
  \makeatletter%
  \providecommand\color[2][]{%
    \errmessage{(Inkscape) Color is used for the text in Inkscape, but the package 'color.sty' is not loaded}%
    \renewcommand\color[2][]{}%
  }%
  \providecommand\transparent[1]{%
    \errmessage{(Inkscape) Transparency is used (non-zero) for the text in Inkscape, but the package 'transparent.sty' is not loaded}%
    \renewcommand\transparent[1]{}%
  }%
  \providecommand\rotatebox[2]{#2}%
  \newcommand*\fsize{\dimexpr\f@size pt\relax}%
  \newcommand*\lineheight[1]{\fontsize{\fsize}{#1\fsize}\selectfont}%
  \ifx\svgwidth\undefined%
    \setlength{\unitlength}{568.67450393bp}%
    \ifx\svgscale\undefined%
      \relax%
    \else%
      \setlength{\unitlength}{\unitlength * \real{\svgscale}}%
    \fi%
  \else%
    \setlength{\unitlength}{\svgwidth}%
  \fi%
  \global\let\svgwidth\undefined%
  \global\let\svgscale\undefined%
  \makeatother%
  \begin{picture}(1,0.21394507)%
    \lineheight{1}%
    \setlength\tabcolsep{0pt}%
    \put(0,0){\includegraphics[width=\unitlength,page=1]{LA_toro.pdf}}%
    \put(-0.00046465,0.08758441){\color[rgb]{0,0,0}\makebox(0,0)[lt]{\lineheight{1.25}\smash{\begin{tabular}[t]{l}$u=\frac{\pi}{2}$\end{tabular}}}}%
    \put(0,0){\includegraphics[width=\unitlength,page=2]{LA_toro.pdf}}%
    \put(0.0029171,0.004794){\color[rgb]{0,0,0}\makebox(0,0)[lt]{\lineheight{1.25}\smash{\begin{tabular}[t]{l}$u=u_1$\end{tabular}}}}%
    \put(-0.00051742,0.18888438){\color[rgb]{0,0,0}\makebox(0,0)[lt]{\lineheight{1.25}\smash{\begin{tabular}[t]{l}$u=u_2$\end{tabular}}}}%
    \put(0,0){\includegraphics[width=\unitlength,page=3]{LA_toro.pdf}}%
    \put(0.53746314,0.09926182){\color[rgb]{0,0,0}\makebox(0,0)[lt]{\lineheight{1.25}\smash{\begin{tabular}[t]{l}$u=\frac{3\pi}{2}$\end{tabular}}}}%
    \put(0,0){\includegraphics[width=\unitlength,page=4]{LA_toro.pdf}}%
    \put(0.54348314,0.01647136){\color[rgb]{0,0,0}\makebox(0,0)[lt]{\lineheight{1.25}\smash{\begin{tabular}[t]{l}$u=u_3$\end{tabular}}}}%
    \put(0.53661409,0.19850105){\color[rgb]{0,0,0}\makebox(0,0)[lt]{\lineheight{1.25}\smash{\begin{tabular}[t]{l}$u=u_4$\end{tabular}}}}%
    \put(0,0){\includegraphics[width=\unitlength,page=5]{LA_toro.pdf}}%
  \end{picture}%
\endgroup%